\documentclass{article}
\usepackage{amssymb}
\usepackage{amsmath,amsthm}
\usepackage{mathtools}
\usepackage{relsize}

\usepackage{framed}
\usepackage{authblk}
\usepackage[numbers,sort]{natbib}

\usepackage{url}
\usepackage{xcolor}

\theoremstyle{definition}

\usepackage{enumitem}
\usepackage{comment}

\colorlet{bscolor}{blue}
\colorlet{skcolor}{red}

\newtheorem{theorem}{Theorem}
\newtheorem{lemma}[theorem]{Lemma}
\newtheorem{definition}[theorem]{Definition}

\newcommand{\bscomment}[1]{\textcolor{bscolor}{BS:#1}}

\newcommand{\Omit}[1]{}

\newcommand{\cf}{CFON}
\newcommand{\chicf}{\chi_{ON}(G)}
\newcommand{\chion}{\chi_{ON}(G)}

\newcommand{\ucn}{U}

\usepackage{array}
\usepackage{tcolorbox}
\tcbuselibrary{breakable}
\newtcolorbox{mybox}[2][]{colbacktitle=white,colback=white,coltitle=black,title={#2},fonttitle=\bfseries,#1, left = 2mm, right = 2mm, breakable}

\begin{document}

\title{A Tight Bound for Conflict-free Coloring 
in terms of Distance to Cluster}
\author{
Sriram Bhyravarapu}
\author{Subrahmanyam Kalyanasundaram}

\affil{Department of Computer Science and Engineering, IIT Hyderabad \\
{\tt \{cs16resch11001,subruk\}@iith.ac.in}}

\maketitle
\begin{abstract}
Given an undirected graph $G = (V,E)$, a conflict-free coloring with respect to open neighborhoods (CFON coloring)
is a vertex coloring 
such that every vertex has a 
uniquely colored vertex in its open neighborhood. The minimum number of colors required for such
a coloring is the CFON chromatic number of $G$, denoted by $\chicf$.

In previous work [WG 2020], we showed the upper bound $\chicf \leq {\sf dc}(G) + 3$, where ${\sf dc}(G)$ denotes the 
distance to cluster parameter of $G$. 
In this
paper, we obtain the improved upper bound of $\chicf \leq {\sf dc}(G) + 1$. We also exhibit a family of graphs 
for which $\chicf > {\sf dc}(G)$, thereby demonstrating that our upper bound is tight. 
\end{abstract}
\section{Introduction}
Given a graph $G=(V,E)$, a \emph{conflict-free coloring} is an assignment of colors to every vertex of $G$ such that there exists a uniquely colored vertex in the 
open neighborhood 
of each vertex. 
This problem was motivated by the frequency assignment problem in cellular networks~\cite{Even2002}, where base stations and clients communicate with each other. 
It is required that there exists a base station with a unique frequency in the neighborhood of each client. 
We formally define the problem as follows.

\begin{definition}[Conflict-Free Coloring]\label{def:cf}
A \emph{\cf{} coloring} of a graph $G = (V,E)$ using $k$ colors is an 
assignment $C:V(G) \rightarrow \{1, 2, \ldots, k\}$ such that for every $v \in V(G)$, 
there exists an $i \in \{ 1, 2, \ldots, k\}$ such that  $|N(v) \cap C^{-1}(i)| = 1$.
The smallest number of colors required for a \cf{} coloring
of $G$ is called the \emph{\cf{} chromatic number of $G$}, denoted by $\chicf$.
\end{definition}

This problem has been studied from both algorithmic and structural perspectives~\cite{smorosurvey, Pach2009, planar, gargano2015, Boe2019, wgarxiv}. Combinatorial bounds on this problem have been studied with respect to vertex cover, treewidth, pathwidth, feedback vertex set and  neighborhood diversity \cite{gargano2015, Boe2019, wgarxiv}. 
In this paper, we study the relation between \cf{} chromatic number and the distance to cluster
parameter, which is formally defined as follows. 
\begin{definition}[Distance to Cluster]
Let $G= (V,E)$ be a graph. The distance to cluster of $G$, 
denoted ${\sf dc}(G)$, is the size of a smallest set $X\subseteq V$  
such that $G[V\setminus X]$ is a disjoint union of cliques. 
\end{definition}
Reddy \cite{vinod2017} showed that $2 {\sf dc}(G) +1$ colors are sufficient to \cf{} color a graph $G$.
This bound has been improved to ${\sf dc}(G)+3$ in \cite{wgarxiv}. In this paper, we 
further improve the bound to ${\sf dc}(G)+1$.  
\begin{theorem}\label{thm:upperbound}
For any graph $G$, we have $\chi_{ON}(G) \leq \max\{3, {\sf dc}(G) + 1\}$. 
\end{theorem}
Further, we show graphs for which ${\sf dc}(G)$ colors are not sufficient, thereby demonstrating 
that the above bound is tight.  
\begin{theorem}\label{thm:lowerbound}
For each value $d \geq 1$, there exist graphs $G$ such that 
${\sf dc}(G) = d$ and $\chion{} > d$.
\end{theorem}

Theorem \ref{thm:lowerbound} is first proved in Section \ref{sec:lowerbound}. The rest of the paper
is devoted to the proof of Theorem \ref{thm:upperbound}. 

\subsection{Preliminaries}
In this paper, we consider only simple, finite, undirected and connected graphs that have at least two vertices. 
If the graph has more than one connected component, Theorem \ref{thm:upperbound} follows 
by its application to each component independently.
Moreover, we assume that $G$ does not have any isolated vertices as 
there is no \cf{} coloring for such graphs.
We denote the set $\{1, 2, \cdots, d\}$ by $[d]$. 
We use the function $C:V\rightarrow [d + 1]$
to denote the color assigned to a vertex. 
The open neighborhood of a vertex $v$, denoted by $N(v)$, is the set of vertices adjacent to $v$. 
The degree of a vertex $v$, denoted $\deg (v)$ is defined as $|N(v)|$. Sometimes, we use the notation
$\deg_A(v) = |N(v) \cap A|$,  where $A\subseteq V$. We use the notation
$G[A]$ to denote the induced graph on the vertex set $A$. We use standard graph theoretic 
terminology from the textbook by Diestel \cite{Diestel}.

During the coloring process, for each vertex $v$, we will designate a vertex as the \emph{uniquely colored neighbor} of $v$, denoted by $\ucn(v)$. The vertex $\ucn(v)$ is a vertex
$w \in N(v)$ such that  $C(w)\neq C(x)$, $\forall x\in N(v)\setminus \{w\}$.
We will also frequently refer to a set $X$ such that $G[V \setminus X]$ is a disjoint 
union of cliques. For the sake of brevity,
instead of referring to a component or maximal clique $K$ of 
$G[V \setminus X]$, we will say ``$K$ is a clique in $G[V \setminus X]$''.

\section{Lower Bound}\label{sec:lowerbound}
In this section, we prove Theorem \ref{thm:lowerbound}.
We will see the existence of graphs $G$ for which 
$\chicf{} >{\sf dc}(G)$. 

\noindent
\textbf{Construction of graph $G$:} Given a positive integer $d$, we construct the graph $G$
such that ${\sf dc}(G) = d$. It consists of three parts as described below.
\begin{itemize}[itemsep=-0.5ex]
    \item The set $X$ is an independent set of $d$ vertices ${v_1, v_2, \cdots, v_d}$. 
          Note that $|X| = d$, and $G[V\setminus X]$ will be a disjoint union of cliques. 
         
    \item Singleton cliques $K_{(i,j)}$,   $\forall 1 \leq i < j \leq d$. 
    For each $K_{(i,j)}$, we have $N(K_{(i,j)}) = \{v_i, v_j\}$. 
    \item A clique $\widehat{K}$ that has $(d+1) 2^d$ vertices. 
          The vertices of the clique $\widehat{K}$ consist of $2^d$ disjoint sets $T^W$, 
          one corresponding to each subset $W \subseteq X$. For each $v \in T^W$, we have $N(v) \cap X = W$.
          Moreover, we have $|T^W| = d + 1$, for each $W \subseteq X$.
\end{itemize}
\begin{proof}[Proof of Theorem \ref{thm:lowerbound}]
It can easily be noted that  ${\sf dc}(G) = d$. We will now see that $\chion{} > d$.

The singleton cliques $K_{(i,j)}$ force 
each vertex in $X$ 
to be assigned a distinct 
color. WLOG, let $C(v_i)=i$ 
for each $v_i\in X$. 
The colors ${1, 2, \cdots, d}$ are used exactly once in $X$. 
Now, we prove that $d$ colors are not sufficient to color the clique $\widehat{K}$. 

We first 
consider a vertex $u_1\in T^\emptyset$. The vertex $u_1$ does not have any neighbors in $X$ and hence
has its uniquely colored neighbor from $\widehat{K}$. Let the vertex $w_1 \in \widehat{K}$ be the
uniquely colored neighbor of $u_1$. WLOG, let $C(w_1)=1$. 
Now, consider the vertices in $T^{\{v_1\}}$. 
At least $d$ vertices in $T^{\{v_1\}}$ 
see both $v_1$ and $w_1$ as their neighbors, and 
hence these vertices cannot have 1 as the unique color in their 
neighborhood. Let $u_2$ be such a vertex. 
WLOG, let $w_2\in C$ be the vertex such that $C(w_2)=2$ 
and $w_2$ acts as the uniquely colored neighbor 
for $u_2$. Of the vertices in $T^{\{v_1, v_2\}}$,
at least $d-1$ of them see $w_1, w_2$ in addition to $v_1, v_2$
as neighbors. Hence these vertices cannot have the colors 1 or 2 
as the unique color in their neighborhood.

We continue this reasoning and show that there exists at 
least one vertex, say $u_{d+1} \in T^X$, 
that sees all the colors $\{1, 2, \ldots, d\}$
at least twice in its neighborhood. 
Hence $u_{d+1}$ cannot have any of the colors 
$1, 2, \ldots, d$ as the unique color in its neighborhood.
Hence we 
require a new color to \cf{} color $G$. 
\end{proof}
\section{Upper Bound}\label{sec:upperbound}

In this section, we prove Theorem \ref{thm:upperbound}. Since it involves several cases and detailed analyses, we first present an overview of the proof, before getting into the details.

\subsection{Overview of the Proof}
Given a graph $G=(V,E)$, and a set of vertices $X\subseteq V$ such that $|X|=d$, we have that $G[V{\setminus} X]$ is 
a disjoint union of cliques. 
We require $d$ colors, one for
each vertex in $X$. 
This is because $G[V\setminus X]$ may contain
$\binom{d}{2}$ singleton cliques, such that each of these 
cliques has degree 2, and adjacent to a pair of vertices in $X$. 
Since a clique can be colored using at most 3 colors, 
it is easy to see that $d+3$ colors are sufficient to \cf{} color $G$ when $G[X]$ is connected. 
Though it is less straightforward, the bound of $d+3$ can be extended
to the case when $G[X]$ is not connected as well \cite{wgarxiv}. 
It is a challenge to further improve the bound to $d+1$.
Our proof requires several cases and subcases 
since there does not seem to be a universal approach 
that leads to a desired coloring. 
The detailed case analysis is necessary because 
of the different forms the induced graph $G[X]$ can take.

Except for some special cases, we will color each vertex in $X$ with a distinct color
from $[d]$.
Our coloring algorithm consists of two phases, an initial phase and a completion
phase. In the initial phase, we color all the vertices of $X$, and identify uniquely colored
neighbors for some vertices in $X$. The key requirement of this phase to identify 
a \emph{free color} $f$, which is a color in $[d]$ that will not serve as a unique color in the neighborhood of any
vertex in $X$. This is straightforward in some cases, like when $G[X]$ has a component 
of size at least $3$. The cases where all the vertices of $G[X]$ have degree 1 (Lemma \ref{lem:perfect_matching}), or all the 
vertices of $G[X]$ have degree 0 or 1 (Lemma \ref{lem:iso_perfect_matching}) prove to be particularly challenging. In some of the cases, this is accomplished by coloring one
or two of the cliques in $G[V\setminus X]$.
The full set of conditions that are to be satisfied by the initial phase
is listed as the Rules of Lemma \ref{lem:isolated_in_X}. 

After the initial phase, we are ready to run the completion phase, which is executed in Lemma \ref{lem:isolated_in_X}. 
The goal of the completion phase is to color the rest of the graph while retaining the 
uniquely colored neighbors of those vertices that had been identified in the intial phase.
In the completion phase, we first identify uniquely colored neighbors for those
vertices in $X$, for which it has not been identified. This involves coloring some of the 
vertices in $V\setminus X$, and hence may partially color 
some of the cliques in $G[V\setminus X]$. The cliques in $G[V\setminus X]$ are colored one by one. 
We have to use different approaches to color them, based on the number of vertices that are already colored in the clique. 
The general results are presented for the case when $d = |X| \geq 3$, and 
the cases $d=1$ and $d=2$ need to be treated differently.

The case $d=1$ is straightforward, but the case $d=2$ is somewhat involved in itself. 
These
cases are presented first in Section \ref{sec:smallx} as they serve as  ``warm-ups'' for the pattern
and the type of arguments that will be used in the general $d \geq 3$ case. However, we 
reiterate that the general case involves a lot of cases and sub-cases that need to be treated
separately and carefully.


\subsection{The case when $|X| < 3$}\label{sec:smallx}
We handle the cases $|X| = 1$ and $|X| = 2$ separately. 

\begin{lemma}\label{lem:xeq1}
Let $G = (V,E)$ be a graph and $X\subseteq V$ be a set of vertices such that 
$|X|=1$ and $G[V\setminus X]$ is a disjoint union of cliques. 
Then $\chi_{ON}(G)\leq 3$. 
\end{lemma}

\begin{proof}
We explain how to assign $C: V  \rightarrow \{1, 2, 3\}$ such that 
$C$ is a CFON coloring of $G$. Let $X = \{v_1\}$. We assign $C(v_1) = 1$. 

\noindent \textbf{Initial phase:} We first choose a clique $K$ in $G[V \setminus X]$. 
Since $G$ is connected, there exists $v \in K$ such that $v_1 \in N(v)$. There 
are two cases.

\begin{itemize}
    \item $|K| = 1$.
    
    That is, $K = \{v\}$. We assign $C(v) = 2$. We have $\ucn(v_1) = v$ and $\ucn(v) = v_1$.
    
    \item $|K| \geq 2$.
    
    We choose a vertex $v' \in K \setminus \{v\}$. We assign $C(v) = 2$, $C(v') = 3$, and
    the vertices (if any) in $K \setminus \{v, v'\}$ are assigned 1. We have $\ucn(v_1) = 
    \ucn(v') = v$, and for all $y \in K \setminus \{v'\}$, we have $\ucn(y) = v'$.
\end{itemize}

\noindent \textbf{Completion phase:} The lone vertex $v_1$ of $X$ is already colored 1
and sees 2 as the unique color in its neighborhood. We have also colored one clique in $G[V \setminus X]$.
For all the uncolored cliques $K$ in $G[V \setminus X]$, we color $K$ as per the applicable case.

\begin{itemize}
    \item All the vertices in $K$ see $v_1$ as their neighbor.
    
    For all $y \in K$, assign $C(y) = 3$. The vertex $v_1$ acts as the uniquely colored
    neighbor for all $y \in K$.
    
    \item There exists $v \in K$ such that $v_1 \notin N(v)$.
    
    Notice that if $|K| = 1$, then the lone vertex in $K$ has to necessarily see $v_1$ as a
    neighbor. Hence in this case, we have $|K| \geq 2$. Choose a vertex $v' \in K \setminus \{v\}$.
    Assign $C(v) = 2$, $C(v') = 3$, and 
    the vertices (if any) in $K \setminus \{v, v'\}$ the color 1.
    
    We have $\ucn(v') = v$ and for all $y \in K \setminus \{v'\}$, we have $\ucn(y) = v'$.
\end{itemize}
\end{proof}

\begin{lemma}\label{lem:xeq2}
Let $G=(V,E)$ be a graph and $X\subseteq V$ be a set of vertices such that 
$|X|=2$ and $G[V\setminus X]$ is a disjoint union of cliques. 
Then $\chi_{ON}(G)\leq 3$. 
\end{lemma}

\begin{proof}
Let $X = \{v_1, v_2\}$. We explain how to assign $C:V\rightarrow \{1,2,3\}$ to get a CFON coloring of $G$.  
There are two cases depending on whether $X$ is an independent set or not. 

\noindent
\textbf{Case 1: $X$ is an independent set. That is, $\{v_1,v_2\}\notin E(G)$.} We have two subcases. 

\begin{itemize}
    \item Every vertex in $V\setminus X$ has at most one neighbor in $X$. 
    
    \textbf{Initial phase:} Since $G$ is connected, there exists a clique $K$ in $G[V \setminus X]$ 
    such that $N(v_1)\cap  K\neq \emptyset$ and $N(v_2)\cap K\neq \emptyset$. 
    Let $v\in N(v_1)\cap K$ and $v'\in N(v_2)\cap K$. Notice that $v\neq v'$, otherwise
    we would be violating the subcase we are in.

    We assign $C(v_1)=C(v_2)=3$, 
    $C(v)=1$, $C(v')=2$ 
    and the remaining vertices (if any) in $K\setminus \{v,v'\}$ the color 3. 
    
    We get that $\ucn(v_1)=v$ and $\ucn(v_2)=v'$.  Also $\ucn(v)=v'$ and for each $y\in K\setminus \{v\}$, $\ucn(y)=v$.  

    \textbf{Completion phase:} Now we color the remaining cliques in $G[V\setminus X]$. Let $K$ be an uncolored clique. There are two possibilities.
    
    \begin{itemize}
        \item Each vertex $y\in K$ has $deg_X(y)=1$. 
        
        For each vertex $y\in K$, if $v_1\in N(y)$, assign $C(y)=2$. 
        Else, assign $C(y)=1$. 
        The uniquely colored neighbor of each vertex in $K$ is its lone neighbor in $X$. 

        \item There exists $v\in K$ such that $\deg_X(v)=0$. 
        
        Since $G$ is connected, there exists $v' \in K$ such that 
        $\deg_X(v') \neq 0$. WLOG let $v_1\in N(v')$. 
        We assign $C(v)=1, C(v')=2$ and the vertices (if any) in $K\setminus \{v,v'\}$ the color 3. 
        
        We get that $\ucn(v)=v'$ and for each $y\in K\setminus \{v\}$, 
        $\ucn(y)=v$. 
    \end{itemize}

    \item There exists a vertex $v\in V\setminus X$ such that $\deg_X(v)=2$.
    
\textbf{Initial phase:} Let $v \in K$, where $K$ is a clique in $G[V \setminus X]$.
We first assign $C(v_1)=1$ and $C(v_2)=2$.

If $|K|=1$, we assign $C(v)=1$ and we get that 
$\ucn(v_1)=\ucn(v_2)=v$ and $\ucn(v)=v_1$. 

Else $|K|\geq 2$ and we have the following cases. 
\begin{itemize}

    \item There exists a vertex $v'\in K\setminus \{v\}$ such that $v_2 \notin N(v')$.

    We assign $C(v)=2$, $C(v')=3$ and the vertices (if any) in $K\setminus \{v,v'\}$ 
    the color 1. 
    We get that $\ucn(v_1)=\ucn(v_2)=v$, 
    $\ucn(v')=v$ and 
    $\ucn(y)=v'$ for all 
    $y\in K\setminus \{v'\}$. 
    
    \item Else, for each $y\in K$, 
    we have $v_2 \in N(y)$. 
    
    We assign $C(v)=1$ and the remaining vertices in 
    $K\setminus \{v\}$ with the color 3. 
    We get that $\ucn(v_1)=\ucn(v_2)=v$ and 
    $\ucn(y)=v_2$ for all $y\in K$. 
\end{itemize}

In each of the above case, both $v_1$ and $v_2$ have the same unique colors
from $\{1, 2\}$. The other color in $\{1, 2\}$ does not serve as the unique color
of $v_1$ and $v_2$ and we refer to it as the \emph{free color}\footnote{
The notion of free color is used crucially in the proof of the general
$|X| \geq 3$ case.}.

\textbf{Completion phase:} WLOG let $v_1$ and $v_2$
have the color 1 as the unique color, and 2 is the free color.

Now, we extend this coloring to the cliques $K\subseteq G[V\setminus X]$.

\begin{itemize}
    \item There exists a vertex $v\in K$ such that $v_1 \in N(v)$.
    
    Assign $C(v) = 3$ and the vertices (if any) in $K \setminus \{v\}$ the color 2.
    We have $\ucn(v) = v_1$ and for all $y \in K \setminus \{v\}$, $\ucn(y) = v$.
    
    \item None of the vertices in $K$ see $v_1$ as a neighbor.
    
    \begin{itemize}
       \item All the vertices in $K$ see $v_2$ as a neighbor.
    
        Assign the color 3 to all the vertices in $K$. For each vertex $y \in K$, we have
        $\ucn(y) = v_2$.
    
        \item There exists $v,v' \in K$ such that $v_2 \in N(v)$ and $v_2 \notin N(v')$.
        
        Assign $C(v) = 3, C(v') = 1$ and the vertices (if any) in $K\setminus \{v,v'\}$  
        the color 2. We get that $\ucn(v) = v'$ and  $y \in K \setminus \{v\}$, $\ucn(y) = v$.
    \end{itemize}
    
\end{itemize}

\end{itemize}

\noindent
\textbf{Case 2: $X$ is not an independent set. That is, $\{v_1, v_2\} \in E(G)$.}

Before we explain how we CFON color the graph, we need to set up notation. A clique
$K\subseteq G[V \setminus X]$ is called \emph{$v_1$-seeing} (\emph{$v_2$-seeing}) if
for all vertices $y \in K$, we have $v_1 \in N(y)$ ($v_2 \in N(y)$).

\begin{itemize}

    \item Each vertex $v\in V\setminus X$ that has $\deg_X(v) = 1$ appears in a clique
    $K$ that is $v_1$-seeing or $v_2$-seeing.
    
    In other words, each clique $K$ in $G[V \setminus X]$ satisfies one of the two conditions:
    (i) $K$ is $v_1$-seeing or $v_2$-seeing, or (ii)
    each vertex $v\in K$ has either $deg_X(v)=0$ or $deg_X(v)=2$. 
    
    \textbf{Initial phase:} There are two cases. 
    
    \begin{itemize}
        \item There is a vertex $v \in V \setminus X$ such that  $\deg_X(v)=2$. 
        
        Let $v \in K$, where $K$ is a clique in $G[V \setminus X]$. 
        We assign $C(v_1) = 1$, $C(v_2) = 2$, $C(v)=3$ and the vertices (if any) in $K \setminus \{v\}$ the color 2.
        We have $\ucn(v_1)=\ucn(v_2)=v$, $\ucn(v)=v_1$ and for the vertices (if any)
        $y\in K\setminus \{v\}$,    $\ucn(y)=v$. 
        
        \item For all $y \in V \setminus X$, we have $\deg_X(y)<2$.
        
        Since $G$ is connected, there is a vertex $v$ in each of the cliques in $G[V\setminus X]$
        such that $\deg_X(v)=1$. 
        By the case definition, we have that each of the cliques must be 
        $v_1$-seeing or $v_2$-seeing. If all the cliques were $v_1$-seeing,
        it follows that $N(v_2) = \{v_1\}$. This means that
        $\mathsf{dc}(G) = 1$ and this case has been addressed
        in Lemma \ref{lem:xeq1}. 
        By an analogous argument, all the cliques cannot be $v_2$-seeing as well.
        Hence  there are cliques $K_1, K_2 \subseteq G[V \setminus X]$
        such that $K_1$ is $v_1$-seeing and $K_2$ is $v_2$-seeing. 
        
        By the case definition, we have that for all vertices
        $y\in K_1$, $\deg_X(y)<2$. Since $K_1$ is $v_1$-seeing, 
        it follows that for all $y\in K_1$, we have $N(y) \cap X = \{v_1\}$. 
        We choose a vertex $v \in K_1$. 
        We assign $C(v_1) = 1$, $C(v) = 3$, and
        the vertices (if any) in $K_1 \setminus \{v\}$, 
        the color 2. We have $\ucn(v_1) = v$, $\ucn(v) = v_1$,
        and for the vertices (if any) $y\in K\setminus \{v\}$,    $\ucn(y)=v$. 
        
        Similarly, for all $y\in K_2$, we have $N(y) \cap X = \{v_2\}$.
        We choose a vertex $v' \in K_2$ and assign $C(v_2) = 2$, $C(v') = 3$, and
        the vertices (if any) in $K \setminus \{v\}$, 
        the color 1. We have $\ucn(v_2) = v'$, $\ucn(v') = v_2$,
        and for the vertices (if any) $y\in K\setminus \{v'\}$,    $\ucn(y)=v'$. 
    \end{itemize}
    
    In each of the above cases, we have $C(v_1) = 1$, $C(v_2) = 2$, and the unique 
    color seen by $v_1$ and $v_2$ is 3.

    \textbf{Completion phase:} Now we color the remaining cliques in $G[V\setminus X]$. Let $K$ be an uncolored clique. There are two possibilities.
     
     \begin{itemize}
         \item For each $y\in K$, we have $\deg_X(y) \geq 1$.
         
         If $K$ is $v_1$-seeing, then for each vertex $y\in K$, assign $C(y)=2$ 
         and we get $\ucn(y)=v_1$. 
         If $K$ is $v_2$-seeing, then for each vertex $y\in K$, assign $C(y)= 1$ 
         and  we get $\ucn(y)=v_2$.  
         (If  $\deg_X(y) = 2$ for each $y \in K$, then either of the above assignments
         work.)
         
         \item There exists a vertex $v\in K$ such that $\deg_X(v) = 0$.
         
         Since $G$ is connected, $K$ has a vertex $v'$ such that $\deg_X(v') > 0$.
         Since $K$ is not $v_1$-seeing or $v_2$-seeing, it must be the case that 
         $\deg_X(v') = 2$.
         We assign $C(v)=3$, $C(v')=1$ and the vertices (if any) in $K \setminus \{v, v'\}$
         the color 2. 
         We get that $\ucn(v)=v'$ and for each $y\in K\setminus \{v\}$, $\ucn(y)=v$. 
     \end{itemize}

    \item There exists a vertex $v\in V\setminus X$ that has $\deg_X(v) = 1$ 
    that appears in a clique $K$ that is neither $v_1$-seeing nor $v_2$-seeing.
    
    \textbf{Initial phase:} WLOG, let $N(v) \cap X = \{v_1\}$. Since $K$ is not $v_1$-seeing, there exists $v' \in K$
    such that $v_1 \notin N(v')$. 
    
    We assign $C(v_1) = 1$, $C(v_2) = 2$, $C(v) = 1$, $C(v') = 3$, and the vertices
     (if any) in $K \setminus \{v, v'\}$ the color 2.  We have $\ucn(v_1) = v$, 
     $\ucn(v_2) = v_1$, $\ucn(v') = v$, and for each $y \in K \setminus \{v'\}$, $\ucn(y) = v'$.
     Note that the unique color seen by both $v_1$ and $v_2$ is 1.
     
    \textbf{Completion phase:} Now we color the remaining cliques in $G[V\setminus X]$. Let $K$ be an uncolored clique. We have the following cases.
    
    \begin{itemize}
        \item There exists a vertex $v\in K$ such that $v_1 \in N(v)$.
        
        Assign $C(v) = 3$, and the vertices (if any) in $K \setminus \{v\}$ the
        color 2. We have $\ucn(v) = v_1$, and for any $y \in K \setminus \{v\}$, $\ucn(y) = v$.
        
        \item None of the vertices in $K$ see $v_1$ as a neighbor.
        
        \begin{itemize}
            \item For all $y \in K$, we have $N(y) \cap X = \{v_2\}$.
            
            Assign the color 3 to all the vertices in $K$. All the vertices in $K$
            see $v_2$ as their uniquely colored neighbor.    
            \item There is a vertex $v\in K$ such that $\deg_X(v) = 0$.
        
            Since $G$ is connected, $|K| \geq 2$. Choose a vertex $v' \in K \setminus \{v\}$.
            Assign $C(v) = 1$, $C(v') = 3$, and the vertices (if any) in $K \setminus \{v, 
            v'\}$ the color 2. We have $\ucn(v') = v$, and for $y \in K \setminus \{v'\}$, $\ucn(y) = v'$.
        \end{itemize}
    \end{itemize}
\end{itemize}
\end{proof}

\subsection{The case when $|X|\geq 3$ and $X$ is an independent set}\label{sec:indset}
We start handling the general case of $|X| \geq 3$. In this section, we
prove the upper bound for the case when $X$ is an independent set. 

\begin{theorem}\label{thm:independent}
Let $G(V,E)$ be a graph and $X\subseteq V$ be a set of vertices such that 
$|X|=d\geq 3$ and $G[V\setminus X]$ is a disjoint union of cliques. 
If $X$ is an independent set, then $\chicf{} \leq d+1$. 
\end{theorem}

In order to show the above theorem, we first prove 
Lemma \ref{lem:D2C_mod_indep_case1}, where we handle the case when every vertex
in $V \setminus X$ has at most one neighbor in $X$. After this, we 
prove Lemma \ref{lem:D2C_mod_indep_case2}, where there is a vertex $v \in V\setminus X$
that has at least two neighbors in $X$. The proof of Lemma \ref{lem:D2C_mod_indep_case2}
uses Lemma \ref{lem:isolated_in_X}, which also serves as the completion phase for all the remaining cases (including those where $X$ is not an independent set).

\begin{lemma}\label{lem:D2C_mod_indep_case1}
Let $G=(V,E)$ be a graph and $X\subseteq V$ be a set of vertices such that 
$|X|=d\geq 3$ and $G[V\setminus X]$ is a disjoint union of cliques. 
If $X$ is an independent set and 
every vertex 
in $V\setminus X$ has at most one neighbor in $X$,
then $\chicf{} \leq d+1$. 

\end{lemma}

\begin{proof}
We explain how to assign $C: V  \rightarrow [d+1]$ such that 
$C$ is a CFON coloring of $G$.
Let $X = \{v_1, v_2, \dots, v_d\}$. 

\noindent\textbf{Initial phase:} For each $v_i\in X$, assign $C(v_i)=d+1$. 
        For each $v_i\in X$, choose an arbitrary neighbor $w_i\in V \setminus X$ 
        and assign $C(w_i)=i$. We get that $\ucn(v_i)=w_i$. 
        Now, each vertex in $X$ is colored and has a uniquely colored neighbor. 
      
\noindent\textbf{Completion phase:} Each uncolored singleton clique in $G[V\setminus X]$ is assigned the color 
        $d+1$.  Note that all the singleton cliques have exactly one neighbor in $X$,
        and this neighbor is the uniquely colored neighbor. 
        What remains to be addressed are cliques of size at least 2.

\begin{itemize}        
        \item \textbf{Clique $K \subseteq G[V\setminus X]$ with at least two colored vertices.} 
        Color the uncolored vertices with the color $d+1$. 
        Let $v,v' \in K$ be two of the vertices that were colored prior 
        to this step. Hence it follows that $C(v), C(v') \in [d]$ and 
        $C(v) \neq C(v')$.
        
        Since $\deg_X(y)\leq 1$ for all $y \in G[V \setminus X]$, one of 
        $v$ and $v'$ will be the uniquely colored neighbor 
        of all vertices in $K$.

        \item \textbf{Clique $K \subseteq G[V\setminus X]$ with exactly 1 colored vertex $v$}. Let $C(v)=j$ and hence $v_j\in N(v)$.   
        \begin{itemize}
                  \item If $|K|=2$. Let $K=\{v,v'\}$. 
                  
                  \begin{itemize}
                      \item If 
                      $v_j\notin N(v')$, 
                      we assign $C(v')=j$. 
                  We get that $\ucn(v)=v'$ and $\ucn(v')=v$. 
                  \item Else, we have $N(v)\cap X = N(v')\cap X = \{v_j\}$. 
                  We assign $C(v')$ arbitrarily from $[d]\setminus \{j\}$. 
                  We get that $\ucn(v)=v'$ and $\ucn(v')=v$. 
                  \end{itemize}

            \item Else if $|K|\geq 3$.

            Let $v'$ be arbitrarily chosen from $K \setminus \{v\}$. 
            Since $|X| \geq 3$ and $\deg_X(v')\leq 1$, there exists a vertex $v_{\ell} \in X$, $v_{\ell} \neq v_j$ such that $v_{\ell} \notin
            N(v') \cap X$. 
            Assign $C(v') = \ell$ and 
            color the rest of the vertices in $K\setminus \{v,v'\}$
            with the color $d+1$. 
            We get that for all vertices $w\in K$, either $v$ or $v'$ is a uniquely colored neighbor. 
            
            \end{itemize}   
            
        \item \textbf{Clique $K \subseteq G[V\setminus X]$ with no colored vertices}. We first select two vertices $v, v' \in K$.
        
        Since $\deg_X(v)\leq 1$, we can choose $j \in [d]$ such that
        $v_j \notin N(v) \cap X$. Since $deg_X(v')\leq 1$ and $|X| \geq 3$,
        we can choose $\ell \in [d]$ such that $\ell \neq j$ and $v_{\ell} 
        \notin  N(v') \cap X$. 
        
        Assign $C(v) = j$, $C(v') = \ell$ and the rest of the vertices in 
        $K\setminus \{v,v'\}$ the color $d+1$. For all vertices $w\in K$, either $v$ or $v'$ is a uniquely colored neighbor. 
\end{itemize}
\end{proof}

We first state Lemma \ref{lem:isolated_in_X}, which will serve as the completion phase
for almost all the remaining cases (even for those where $X$ is not an independent set).
This lemma states that the graph can be \cf{} colored provided it has been  
partially colored satisfying certain rules. 

\begin{lemma}\label{lem:isolated_in_X}
Let $G = (V,E)$ be a graph with $X = \{v_1, v_2, \ldots, v_d\} \subseteq V$ such that $d\geq 3$ 
and $G[V \setminus X]$ is a disjoint union of cliques. Further,
$Y = \{v_i \in X: \deg_X(v_i) \geq 1\}$ and $C: V \rightarrow [d+1]$ 
be a partial coloring that satisfies the below rules.

Then, $C$ can be extended to
a CFON coloring $\widehat C: V \rightarrow [d+1]$ of all the vertices in $V$.
\end{lemma}

\begin{mybox}{Rules}
\begin{enumerate}[topsep=0pt,itemsep=-0.5ex,partopsep=1ex,parsep=1ex, label=(\roman*)]
    \item For all $v_i\in X$, $C(v_i)=i$.

    \item For some number of cliques $K$ in $G[V \setminus X]$, 
    all the vertices in $K$ are colored using colors from $[d+1]$. The remaining cliques are uncolored.
  
  \item All the vertices in 
   $Y$ and all the colored vertices in $V\setminus X$ have a uniquely colored neighbor. 
   Moreover, some vertices in $X \setminus Y$ have a uniquely colored neighbor.
   
   \item The uniquely colored neighbor is identified for all the vertices in $X$ whose 
   entire neighborhood is colored. 
   
    
    \item There exists $1\leq f \leq d$, such that (a) $v_f \in X$ has a uniquely colored
    neighbor and (b) for  each vertex in $X$, the color $f$ is not the unique color in its neighborhood. 
    We refer to $f$ as the \emph{free color}.
    
    \item If a vertex $v_i\in X\setminus Y$ does not have a uniquely colored neighbor, 
    then the color $i$ is not assigned to 
    any vertex in $V\setminus X$ yet. 
%
    
\end{enumerate}
\end{mybox}

Before proving Lemma \ref{lem:isolated_in_X}, we prove the upper bound when 
$X$ is an independent set and there is a vertex $v \in V \setminus X$ that has at
least two neighbors in $X$. 

\begin{lemma}\label{lem:D2C_mod_indep_case2}
Let $G(V,E)$ be a graph and $X\subseteq V$ be a set of vertices such that 
$|X|=d\geq 3$ and $G[V\setminus X]$ is a disjoint union of cliques. 
If $X$ is an independent set and 
there exists a vertex $v\in V\setminus X$, such that 
$\deg_X(v)\geq 2$, then $\chicf{} \leq d+1$. 

\end{lemma}

\begin{proof}
The goal here is to partially color some vertices of $G$
so that the rules of Lemma \ref{lem:isolated_in_X} are satisfied.
We then use Lemma \ref{lem:isolated_in_X} to extend the partial coloring and obtain a CFON coloring of $G$.

Let $X = \{v_1, v_2, \dots, v_d\}$. 
We explain how to assign $C: V  \rightarrow [d+1]$ such that 
$C$ is a partial coloring that satisfies the rules of Lemma \ref{lem:isolated_in_X}.
For each vertex $v_i\in X$, we assign a distinct color $C(v_i)=i$. 
There are two cases depending on the neighborhood of 
vertices in $V\setminus X$.

\begin{itemize}
   \item \textbf{There exists a singleton clique $K=\{v\}$ such that 
    $\deg_X(v)\geq 2$}. 
   
   Let $N(v)\cap X = \{ v_{i_1}, v_{i_2}, \dots , v_{i_m} \}$, with $m\geq 2$. 
    We assign $C(v)=i_1$. 
    
    We get that $\ucn(v)=v_{i_1}$ and for all $1 \leq j \leq m$, $\ucn(v_{i_j})=v$. 
    The color $i_2$ will not be the unique color of any 
        vertex in $X$ and will be the free color.

    \item \textbf{All singleton cliques have degree equal to 1.}
     By assumption, there is 
      a vertex $v\in V\setminus X$, such that 
$\deg_X(v)\geq 2$. It follows that $v \in K$ where 
     $K$ is a clique in  $G[V \setminus X]$ and 
    $|K|\geq 2$. 
    Let $N(v)\cap X = \{v_{i_1}, v_{i_2}, \dots , v_{i_m}\}$, with $m\geq 2$.

    Now we have two cases depending on whether there exists vertices in $X$ 
    whose neighbors belong to only $K$. 
    We refer to these vertices as $S_K$. Formally, $S_K = \{
    v_i \in X : N(v_i) \subseteq K\} \setminus N(v)$. 
    The vertices in $N(v) \cap X$ rely on $v$ for their uniquely colored 
    neighbor and hence does not require special attention.

  \begin{itemize}
      \item $S_K \neq \emptyset$. 

    First, we assign $C(v)=i_1$   and choose $i_2$ as the free color.     
   For each vertex $v_j\in S_K$, we choose a vertex $w_j \in N(v_j)$. 
   Note that by the definition of $S_K$, it follows that $w_j \in K$.
   We assign $C(w_j)=j$ if it is not already assigned ($w_j$ could have been the 
   chosen neighbor for some other vertex $v_{j'}$ as well). Now all the
   vertices in  $S_K$ have a uniquely colored neighbor.
   
  \begin{itemize}
      \item If all vertices in $K$ are colored because of the above coloring, 
      every vertex in 
      $K$ has a uniquely colored neighbor. 
      We get that $\ucn(v) = v_{i_1}$.
      Each vertex $w\in K\setminus \{v\}$ is assigned a distinct color, say $j$, because it is adjacent to 
      $v_j \in S_K$, which serves as its uniquely colored neighbor.
   
   
   \item There exists at least one uncolored vertex $K$.  
   \begin{itemize}
       \item  If there exists a uniquely colored neighbor for each uncolored vertex
    in $K$, assign the color $d+1$ to all the uncolored vertices.
    We get that $\ucn(v) = v_{i_1}$. The vertices $w_j \in K \setminus \{v\}$
    rely on the corresponding $v_j$'s as mentioned above.
   
   
   \item Else, let $v' \in K$ be an uncolored vertex that does not 
    see a uniquely colored neighbor. 
    This implies that 
    $N(v') \cap \{v_{i_1}, v_{i_2}, \cdots, v_{i_m}\}  = \{v_{i_1}\}$.
    We reassign $C(v)=i_2$, assign $C(v')=d+1$, and 
    designate $i_1$ as the free color instead of $i_2$.
    We assign the color $i_1$ to the remaining uncolored vertices in $K$.
    We have $\ucn(v')=v$ and $\ucn(w)=v'$, for all 
    $w\in K\setminus \{v'\}$. 
    \end{itemize}

    

    
    
  \end{itemize} 
  
  \item $S_K=  \emptyset$. 

This implies that, 
there is no vertex in $X\setminus \{v_{i_1}, v_{i_2}, \cdots, v_{i_m}\}$  that was relying on $K$ for its uniquely colored neighbor. 
We first assign $C(v)=i_1$   and choose $i_2$ as the free color. 
We have the following cases.    
    \begin{itemize}

    \item There exists a vertex $v'\in K\setminus \{v\}$ 
    such that 
     $v_{i_1}\notin N(v')$. 

    Assign $C(v')=d+1$ and assign the remaining vertices of $K\setminus \{v,v'\}$ 
     the color $i_2$.

    For every vertex $w\in K\setminus \{v'\}$, $\ucn(x)=v'$. 
    Finally, we have $\ucn(v')=v$.  
    
    \item 
Else, for every vertex $w \in K$, we have $v_{i_1} \in N(w) \cap X$.

    Reassign $C(v) = i_2$ and assign the color $d+1$ to all the 
    vertices in $K \setminus \{v\}$. 
    The color $i_1$ is the redesignated free color. 
    
    And for each $w\in K$, $\ucn(w)=v_{i_1}$. 
  \end{itemize}

    \end{itemize}
\end{itemize}

Now $C$ is a partial color assignment satisfying all the conditions in the 
rules of Lemma \ref{lem:isolated_in_X}.
Hence by Lemma \ref{lem:isolated_in_X}, we can extend $C$ to a full CFON coloring
of $G$ that uses at most $d+1$ colors. 
\end{proof}

We conclude this section with the proof of Lemma \ref{lem:isolated_in_X}.

\begin{proof}[Proof of Lemma  \ref{lem:isolated_in_X}]
For each colored vertex $w\in G$, $\widehat C (w) =C(w)$. 
We explain how to extend $\widehat C: V  \rightarrow [d+1]$ 
to all vertices 
such that 
$\widehat C$ is a CFON coloring of $G$. 
Let $X = \{v_1, v_2, \dots, v_d\}$ and $Y = \{v_i \in X : \mbox{deg}_X (v_i)\geq 1\}$. 



\noindent \textbf{Process to identify uniquely colored neighbors for $X\setminus Y$:}  For every $v_j\in X\setminus Y $, 
that does not have a uniquely colored neighbor, 
choose an 
uncolored neighbor of $v_j$, say $w_j\in V\setminus X$ 
and assign $\widehat  C(w_j)=j$. 
Rules (iv) and (vi) of Lemma \ref{lem:isolated_in_X} allow us to do this.
Since $v_j$ does not have a uniquely colored neighbor, rule (iv) implies that 
$v_j$ has an uncolored neighbor, and as per rule (vi), no vertex in 
$V\setminus X$ is assigned the color $j$.

\noindent \textbf{Observation:} It is possible that all the neighbors of $v_j\in X\setminus Y$ may be colored 
by the above coloring process on other vertices $v_i \in X \setminus Y$
even before applying the process on $v_j$.
In such a case, we choose an arbitrary neighbor of $v_j$
that was already colored by this process and 
assign it as the uniquely colored neighbor for $v_j$. 
This neighbor 
acts as the  uniquely colored neighbor 
for at least 2 vertices in $X\setminus Y$. 
This fact will be useful later. 
    
Now, every vertex in $X\setminus Y$ has a uniquely colored 
neighbor. We now look at the previously uncolored cliques $K$ in $G[V\setminus X]$. 
For each such clique $K$, we  color $K$ as per the applicable case below.

\noindent
\textbf{Case 1: $K$ has no colored vertices}
 \begin{itemize}[topsep=-0.25pt,itemsep=-0.5ex,partopsep=1ex,parsep=1ex]
            \item $|K|=1$. Let $K = \{w\}$. 
            
            We assign $\widehat C(w) = d+1$. As all the 
            neighbors of $w$ are distinctly colored, we assign one of the neighbors
            as $\ucn(w)$.
            
            \item $|K|\geq 2$. We have two subcases here.
            \begin{itemize}
                \item There exists a vertex $w\in K$, such that 
            $N(w)\cap X= \emptyset$.
            
            Choose another vertex $w'\in K\setminus \{w\}$
            such that $N(w') \cap X \neq \emptyset$.
            We have two subcases.
            \begin{itemize}
                \item $N(w') \cap X = \{ v_f\}$, where $v_f\in X$ is the vertex that corresponds to the free color $f$.
                
                Assign $\widehat C(w') = d+1$ and $\widehat C(w) = c$, where
                $c \in [d] \setminus \{f\}$, chosen arbitrarily. 
                For all the vertices (if any) $x \in K \setminus \{w, w'\}$, 
                assign $\widehat C(x) = f$.
                
                We have $\ucn(w') = w$ and for all vertices $x\in K\setminus \{w'\}$, we have $\ucn(x)=w'$. 
                
                \item There exists a vertex $v_i \in N(w') \cap X$, where $v_i \neq v_f$.
                
                Assign $\widehat C(w')=d+1$.
                For all the vertices $x \in K \setminus \{w'\}$, assign $\widehat C(x) = f$.
            
                 We have  $\ucn(w') = v_i$ and for all vertices $x\in K\setminus \{w'\}$, we have $\ucn(x)=w'$. 
                
            \end{itemize}
            
            \item For all $w \in K$, $N(w) \cap X \neq \emptyset$. 
            
            Assign all the vertices in $K$ the color $d+1$. 
            For all the vertices in $K$, we assign one of the 
            neighbors in $X$ as the respective uniquely colored neighbor.
            \end{itemize}
            
        \end{itemize}

\noindent
\textbf{Case 2: $K$ has exactly one colored vertex}

WLOG, let $v\in K$ be such that $\widehat C(v)=j$. This implies that 
$v_j\in N(v)\cap X$. 

\begin{itemize}[topsep=-0.25pt,itemsep=-0.5ex,partopsep=1ex,parsep=1ex]
    \item $|K|=1$. 
    
    We have $\ucn(v)=v_j$ and $\ucn(v_j) = v$ (as was already assigned).
    \item $|K|=2$. Let $K=\{v,v'\}$. 
    
    \begin{itemize}
        \item $N(v')\cap X=\emptyset$. 
        
        We assign 
        $\widehat C(v')=d+1$. 
        We get that $\ucn(v)=v'$, $\ucn(v')=v$ and $\ucn(v_j)=v$.
        \item  $v'$ has a neighbor other than $v_j$ in $X$. 
        That is,  $\exists v_k\in N(v')\cap X$, with $v_k \neq v_j$. 
        
        We assign $\widehat C(v')=d+1$. 
        We get that $\ucn(v)=v'$, $\ucn(v')=v_k$ and $\ucn(v_j)=v$.
        
        \item $N(v') \cap X = \{v_j\}$.
        \begin{itemize}
            \item There exists a vertex 
            $v_\ell\in X\setminus Y$, $v_{\ell} \neq v_j$ 
            such that $\ucn(v_{\ell}) = v$.
            
            We reassign $\widehat C(v)=\ell$ and $\widehat C(v')=d+1$.
            We have that $\ucn(v)=v'$ and $\ucn(v')=v$.
            Note that $\ucn(v_j)= \ucn(v_\ell)=v$, as before.

        \item No vertex in $X$ other than $v_j$ sees $v$ as its uniquely colored neighbor and 
         $v$ has another neighbor in $X$ besides $v_j$, say $v_k$.

                We reassign $\widehat C(v)=d+1$ and  assign $\widehat C(v')=j$.
                We get that $\ucn(v')=v$,  $\ucn(v)=v_k$ and we reassign $\ucn(v_j)=v'$.

        \item $N(v) \cap X = N(v') \cap X = \{v_j\}$.
        
        We reassign $\widehat C(v)$ to an arbitrarily chosen value from $[d]\setminus\{j, f\}$. 
        Note that such a value exists since $d \geq 3$.
        We assign $\widehat C(v')=d+1$. We get that $\ucn(v)=v'$, 
        $\ucn(v')=v$ and $\ucn(v_j)=v$, as before.
        \end{itemize}

    \end{itemize}
    
    \item $|K|\geq3$. 
    
    \begin{itemize}
        \item There exists a vertex $v'\in K\setminus \{v\}$ such that 
    $v_j\notin N(v')$.
    
    We assign 
    $\widehat C(v')=d+1$ and the vertices in $K \setminus \{v, v'\}$ are colored with the free color $f$. 
    
    We get that $\ucn(v')=v$ and for all $w\in K\setminus \{v'\}$, $\ucn(w)=v'$.
    As before, $\ucn(v_j)=v$.
    
    \item Every vertex in $K$ is adjacent to $v_j$. 
    \begin{itemize}
        \item There exists a vertex $v'\in K\setminus \{v\}$ such that 
        $(N(v') \cap X) \setminus~\{v_j, v_f\} \neq \emptyset$.
        Let $v_k \in N(v') \cap X$, where $v_k \neq v_j$ and $v_k \neq v_f$.
    
        Assign 
        $\widehat C(v')=d+1$ and rest of the vertices in $K\setminus \{v,v'\}$ are 
        colored with the free color $f$. 
        We get that $\ucn(v')=v_k$ and for all $w\in K\setminus \{v'\}$, $\ucn(w)=v'$. As before, $\ucn(v_j)=v$.
        
        \item For all $w \in K\setminus \{v\}$, $N(w) \cap X \subseteq \{v_j, v_f\}$.
        
        We choose $k \in [d]\setminus\{j, f\}$ arbitrarily. 
        Note that such a value exists since $d \geq 3$.
        
        Choose two vertices $v', v''$ arbitrarily from $K\setminus \{v\}$ and assign $\widehat C(v')=d+1$, $\widehat C(v'')=k$ 
        and the remaining vertices in $K\setminus \{v,v', v''\}$ are colored with $f$. 
        
        We get that $\ucn(v')=v''$, and for all $w\in K\setminus \{v'\}$, 
        $\ucn(w)=v'$. As before $\ucn(v_j)=v$.
        
    \end{itemize}

    \end{itemize}
    \end{itemize}
    \noindent
    \textbf{Case 3: $K$ has at least two colored vertices and there exists a 
    vertex in $K$ that is a uniquely colored neighbor for 
    at least two vertices in $X\setminus Y$} 
    
    Let $v'\in K$ be such that $\ucn(v_j) = \ucn(v_k) = v'$, where $v_j, v_k \in X \setminus Y$. The vertex $v'$ was colored in the process to identify
    uniquely colored neighbors for vertices in $X\setminus Y$. WLOG, we may assume that
    $v'$ was colored when assigning a unique colored neighbor for $v_j$. That is,
    $\widehat C(v') = j$. This also implies that no vertices in $K$ are colored $k$. 
    There are two cases here.
    
    \begin{itemize}[topsep=-0.25pt,itemsep=-0.5ex,partopsep=1ex,parsep=1ex]
        \item All vertices in $K$ are colored. 
        
        In this case, every vertex in $K$ will have a uniquely colored neighbor.
        This is because every vertex in $K$ would have been assigned a distinct color.
        If $w \in K$ is such that  $\widehat C(w) = \ell$, then $\ucn(w) = v_{\ell} \in X \setminus Y$.

    \item There exists an uncolored vertex $v\in K$. There are two subcases.
    \begin{itemize}
        \item $v$ is adjacent to $v_k$. 
        
        We assign $\widehat C(v)=d+1$ and the remaining uncolored vertices 
        in $K$ (if any) are assigned the free color $f$. 
        
        We get that $\ucn(v)=v_k$, and for all $w\in K\setminus \{v\}$,   $\ucn(w)=v$.
        Note that $\ucn(v_j)= \ucn(v_k) = v'$, as before.   
        \item $v$ is not adjacent to $v_k$. 
        
        Reassign $\widehat C(v')=k$, assign $\widehat C(v)=d+1$ and the remaining uncolored 
        vertices in $K$ (if any) are assigned the free color $f$. 
        
        We get that  $\ucn(v)=v'$, and for all $w\in K\setminus \{v\}$, 
        $\ucn(w)=v$. Note that $\ucn(v_j)= \ucn(v_k) = v'$, as before. 
    \end{itemize}
      \end{itemize}
    
    \noindent \textbf{Case 4: $K$ has at least two colored vertices and every colored vertex in $K$ is the uniquely colored neighbor for exactly one vertex in $X\setminus Y$}
    
    Let 
        $v,v'\in K$ be two colored vertices  
        such that $\widehat C(v)=j$ and $\widehat C(v')=k$. 
        The colors $j$ and $k$ are assigned because they are adjacent to $v_j$ and $v_k$ respectively, where $v_j, v_k \in X\setminus Y$. 
        We have cases depending on the neighborhood of $K$.

        \begin{itemize}[topsep=-0.25pt,itemsep=-0.5ex,partopsep=1ex,parsep=1ex]
            \item There exists a colored vertex in $K$ that is adjacent to both $v_j$
            and $v_k$.
            \begin{itemize}
            \item At least one of $v$ or $v'$ is adjacent to 
            both $v_j$ and $v_k$. 
           
            WLOG let that vertex be $v$. 
        Reassign $\widehat C(v')=d+1$ and assign the color $f$ to  
        the remaining uncolored vertices (if any).
        
        We have that $\ucn(v')=v_k$, and  for all vertices $w\in K\setminus \{v'\}$, $\ucn(w)=v'$. We reassign $\ucn(v_k) = v$, while $\ucn(v_j) = v$ as before. 
            \item There exists a colored vertex $v'' \in K\setminus \{v,v'\}$  such that 
            $\{v_j,v_k\} \subseteq N(v'')$. 
            
            Let $\widehat C(v'') = \ell$ because it was the chosen neighbor for
            $v_{\ell} \in X \setminus Y$ in the coloring process stated in the beginning
            of this proof.
            We reassign $\widehat C(v)=d+1$ and $\widehat C(v')=f$.
            It is important to note that because of the case definition, this reassignment
            does not affect the uniquely colored neighbors of vertices in $X \setminus (Y \cup \{v_j, v_k\})$.
            
            The remaining uncolored vertices in $K$ (if any) are assigned the color $f$.
            
            We have $\ucn(v)=v_j$, and for every vertex $w\in K\setminus \{v\}$, $\ucn(w)=v$. 
             We reassign $\ucn(v_j)=\ucn(v_k) = v''$, while $\ucn(v_{\ell}) = v''$ as before.
            
            \end{itemize}
        \item There exists an uncolored vertex $v''\in K$ such that $\{v_j,v_k\} 
        \subseteq N(v'')$. 
        
        We reassign $\widehat C(v)=d+1$, 
        $\widehat C(v')=f$ and 
        assign $\widehat C(v'')=k$. 
        The remaining uncolored vertices in $K$ (if any) are assigned the color $f$. 
       
        We have $\ucn(v)=v_j$, and for every vertex $w\in K\setminus \{v\}$, $\ucn(w)=v$. 
        We reassign $\ucn(v_j)=\ucn(v_k) = v''$.
        \item No vertex in $K$ is adjacent to both $v_j$ and $v_k$. 
        
        Assign the color $f$ to the remaining uncolored vertices (if any). 
        Every vertex in $K\setminus\{v,v'\}$ will see either $v$ or $v'$ as its uniquely colored neighbor. 
        Also $\ucn(v)=v_j$ and $\ucn(v')=v_k$, while $\ucn(v_j)=v$ and $\ucn(v_k)=v'$, as before. 
\end{itemize}
\end{proof}

\subsection{The case when $|X|\geq 3$  and $X$ is not an independent set}\label{sec:notindset}

In this section, we prove the upper bound when $|X|\geq 3$ and $X$ is not an independent set.
\begin{theorem}\label{thm:D2Csomeedges}
Let $G = (V,E)$ be a graph and $X\subseteq V$ be a set of vertices such that 
$|X|=d\geq 3$ and $G[V\setminus X]$ is a disjoint union of cliques. 
 If $X$ is not an independent set, then 
 $\chicf{} \leq d+1$.
\end{theorem}

The proof of Theorem \ref{thm:D2Csomeedges} involves a lot of cases. The cases when 
$G[X]$ is 1-regular and all the vertices in $G[X]$ have degree 0 or 1 needs particular care.
We state these two cases below. The proofs of Lemmas \ref{lem:perfect_matching} and \ref{lem:iso_perfect_matching}, are proved in 
Sections \ref{sec:perfectmatching} and \ref{sec:zerone} respectively. 
All the proofs in this section and subsequent sections will only deal with the initial
 phase, i.e., to achieve a partial coloring of the graph that satisfies the Rules
of Lemma \ref{lem:isolated_in_X}. The completion phase follows by 
Lemma \ref{lem:isolated_in_X}.

\begin{lemma}\label{lem:perfect_matching}
Let $G = (V,E)$ be a graph and $X\subseteq V$ be a set of vertices such that 
$|X|=d\geq 3$ and $G[V\setminus X]$ is a disjoint union of cliques. 
If $G[X]$ is 1-regular (perfect matching), 
then $\chi_{ON}(G)\leq d+1$. 
\end{lemma}

\begin{lemma}\label{lem:iso_perfect_matching}
Let $G=(V,E)$ be a graph and $X\subseteq V$ be a set of vertices, such that $|X|=d\geq 3$ and 
$G[V\setminus X]$ is a disjoint union of cliques. 
Moreover all the vertices in $G[X]$ 
have degree at most 1 and 
at least one vertex has degree 0. 
Then 
$\chi_{ON}(G)\leq d+1$. 
\end{lemma}

We first prove Theorem \ref{thm:D2Csomeedges} assuming the above lemmas.

\begin{proof}[Proof of Theorem \ref{thm:D2Csomeedges}]
Let $X = \{v_1, v_2, \dots, v_d\}$ and $Y = \{v_i \in X : \mbox{deg}_X (v_i)\geq 1\}$. 
Also, let $\mathcal A$ be the set of connected components of $G[X]$.

For each vertex $v_i\in X$, we assign a distinct color $C(v_i)=i$. We have the following cases depending 
on the components. 

\begin{itemize}[topsep=-0.25pt,itemsep=-0.5ex,partopsep=1ex,parsep=1ex]
    \item There exists a component $A \in \mathcal A$ such that $|A|\geq 3$, and $\exists v_j\in A$  with $\deg_X(v_j)=1$.

        Let $N(v_j)\cap X = \{v_k\}$. 
        Since $|A|\geq 3$, there exists $v_{\ell} \in N(v_k) \cap X$ such that 
        $v_{\ell} \neq v_j$.
        Every vertex in $A\setminus\{v_k\}$ chooses an arbitrary neighbor in $A$ as its 
        uniquely colored neighbor while we assign $\ucn(v_k) = v_{\ell}$.
        The color $j$ is not the unique color for 
        any vertex in $X$. 
        Hence we use $j$ as the free color for the rest of the coloring. 
        
        For all $A'\in \mathcal{A}\setminus {A}$, 
        where $|A'|\geq 2$, a vertex chooses one of its neighbors in $A'$ as its uniquely colored neighbor. Thus all the vertices in $Y$ have a uniquely colored neighbor.
        
        At this point, we have partially colored $G$ satisfying the rules of Lemma \ref{lem:isolated_in_X}. Using Lemma \ref{lem:isolated_in_X}, we can extend this to a full CFON coloring of $G$ that uses $d+1$ colors.

    \item There exists a $A\in \mathcal A$ such that 
    $|A|\geq 3$, and all the vertices in $A$ have degree at least 2 in $G[X]$. 

         Choose a vertex $v_j\in A$. Since every vertex in $A$ has degree at least 2,
         it can be ensured that every vertex in $A$ is assigned a uniquely colored
         neighbor other than $v_j$. 
         We use the color $j$ as the free color.

        For all $A'\in \mathcal{A}\setminus {A}$, 
        where $|A'|\geq 2$, a vertex chooses one of its neighbors in $A'$ as its uniquely colored neighbor. Thus all the vertices in $Y$ have a uniquely colored neighbor.
        
        Since the rules of Lemma \ref{lem:isolated_in_X} are satisfied, we can extend 
        this coloring to a CFON coloring of $G$ that uses $d+1$ colors.
    
    \item For all the components $A\in \mathcal A$, we have   $|A|= 2$. 

     In this case,  we have $X\setminus Y=\emptyset$. 
    All the vertices $v_i\in X$ have $deg_X(v_i)=1$. 
    We apply Lemma \ref{lem:perfect_matching} to CFON color $G$ with $d+1$ colors. 
 
    \item For all the components $A\in \mathcal A$, we have   $|A|\leq 2$. Moreover
    there exists $A' \in \mathcal A$ such that $|A'| = 1$.
    
    That is $X \setminus Y \neq \emptyset$. By assumption, $X$ is not an independent set.
    Hence $Y \neq \emptyset$ as well. We apply Lemma \ref{lem:iso_perfect_matching} to CFON color $G$ with $d+1$ colors. 
        \end{itemize}
\end{proof}

\subsection{Proof of Lemma \ref{lem:perfect_matching}}\label{sec:perfectmatching}



Let $X = \{v_1, v_2, \dots, v_d\}$. 
Since each vertex $v_i \in X$ has $\deg_X(v_i)=1$, we have that $d = |X|$ is an even number. 
This implies that $d \geq 4$.
WLOG, we may assume that the edges in $G[X]$ are $\{v_1, v_2\}, \{v_3, v_4\}, \ldots, \{v_{d-1}, v_d\}$.
We explain how to assign $C: V  \rightarrow [d+1]$ such that 
$C$ is a partial coloring that satisfies the rules of Lemma \ref{lem:isolated_in_X}.

For each vertex $v_i\in X$, we assign the color $C(v_i)=i$. 
We have the following cases.

\noindent
\textbf{Case 1:} There exists a vertex $v\in V\setminus X$
such that $\deg_X(v)=|X|$. 
    
Let $v \in K_1$, where $K_1$ is a clique in $G[V \setminus X]$.



\noindent
\textbf{Subcase 1.1:} $K_1$ is the only clique in 
$G[V\setminus X]$.

    For all $v_i\in X\setminus \{v_1\}$, we reassign 
    $C(v_i)=d+1$. 
    Assign $C(v)=2$ and the 
    remaining vertices (if any) in $K_1\setminus \{v\}$ are assigned the color $d+1$. 
    
    We get that $\ucn(v)=v_1$ and for all $x\in V\setminus \{v\}$, $\ucn(x)=v$. 
    Thus the entire graph is CFON colored.

\noindent
\textbf{Subcase 1.2:} There exists a clique $K_2 \neq K_1$, such that 
$K_2=\{w\}$ and 
$\deg_X(w)\geq 2$. 
\begin{itemize}

                \item 
              \textbf{$N(w)\cap X$ contains a pair of adjacent vertices.} 
           
            WLOG, let $v_1,v_2\in N(w)$. 
            We have cases based on the size of the clique $K_1$.  
            \begin{itemize}
                \item $|K_1|=1$. 
                
                We assign $C(w)=3$ and $C(v)=1$. 
                            We have that
                $\ucn(v)=\ucn(w)=v_1$.
                Also,  $\ucn(v_2)=w$ and  for each $v_i\in X\setminus \{v_2\}$, $\ucn(v_i)=v$. 
                We have color 4 as the free color. 
                
                \item $|K_1|\geq 2$. 
               
                We consider an arbitrary vertex $v'\in K_1\setminus \{v\}$. We have three cases based on the neighborhood of $v'$.

                \begin{itemize}
                    \item $v_1\notin N(v')$.  
                
                We assign $C(w)=3$, 
                $C(v)=1$, $C(v')=d+1$ and the vertices (if any) 
                in $K\setminus\{v,v'\}$ the color 4. 
                
                We have that  
                $\ucn(w)=v_2$, $\ucn(v_2)=w$, 
                and for each $v_i\in X\setminus \{v_2\}$, $\ucn(v_i)=v$. 
                Also $\ucn
                (v')=v$ and for each $y\in K_1\setminus \{v'\}$, $\ucn(y)=v'$.
                We have color 4 as the free color.




                \item $v_1 \in N(v')$.

                We assign $C(w)=3$, 
                $C(v)=2$, $C(v')=d+1$ and the 
                 vertices (if any) in $K_1\setminus \{v,v'\}$ the color 4.

                 We have that  
                $\ucn(w)=v_2$, $\ucn(v_1)=w$
                and for each $v_i\in X\setminus \{v_1\}$, $\ucn(v_i)=v$. 
                Also $\ucn
                (v')=v_1$ and for each $y\in K_1\setminus \{v'\}$, $\ucn(y)=v'$.
                We have color 4 as the free color.

                

                \end{itemize}
    
            \end{itemize}

                \item \textbf{None of the vertices in $N(w) \cap X$ are adjacent to each other.} 
                
                WLOG, let $v_1, v_3\in N(w)$. 
                 We have cases based on the size of the clique $K_1$.  

              \begin{itemize}
                  \item $|K_1|=1$.  
                  
                  We assign $C(v)=4$ and $C(w)=1$. 
                We have $\ucn(w) = \ucn(v) = v_1$. Also, $\ucn(v_3) = w$,
                and for each $v_i\in X\setminus \{v_3\}$, $\ucn(v_i)=v$. 
                We have color 3 as the free color. 
                
                \item $|K_1|\geq 2$. 
                
                We consider an arbitrary vertex $v'\in K_1\setminus \{v\}$. We have two cases based on the neighborhood of $v'$.        
                
    \begin{itemize}
    
         \item $v_2\notin N(v')$. 
                
                We assign $C(w)=3$, 
                $C(v)=2$, $C(v')=d+1$ and the
                vertices (if any) in $K_1\setminus \{v,v'\}$ the 
                color 4. 
                
                We have that $\ucn(w)=v_1$, $\ucn(v_1)=w$, 
                and for each $v_i\in X\setminus \{v_1\}$, $\ucn(v_i)=v$. 
                Also $\ucn
                (v')=v$ and for each $y\in K_1\setminus \{v'\}$, $\ucn(y)=v'$.
                Color 4 is the free color.

        \item  $v_2\in N(v')$. 
        
                We assign $C(w)=1$,
                $C(v)=4$,  $C(v')=d+1$ and the 
                 vertices (if any) in $K_1\setminus \{v,v'\}$ the 
                color 3. 

                We have that $\ucn(w)=v_1$, $\ucn(v_3)=w$, 
                and for each $v_i\in X\setminus \{v_3\}$, $\ucn(v_i)=v$. 
                Also $\ucn
                (v')=v_2$ and for each $y\in K_1\setminus \{v'\}$, $\ucn(y)=v'$.
                Color 3 is the free color.

    \end{itemize}

              \end{itemize}

 \end{itemize}           
    
    \noindent
    \textbf{Subcase 1.3:}  There exists a clique $K_2 \neq K_1$, such that 
$K_2=\{w\}$ and 
$\deg(w)= 1$.

WLOG let $N(w)=\{v_1\}$. We have cases based on the size of the clique $K_1$. 
 
     \begin{itemize}
           
            \item $|K_1|=1$. 

            We assign $C(v)=2$ and $C(w)=3$. 
            We have $\ucn(v) = \ucn(w) = v_1$. 
            Also, $\ucn(v_1) = w$,
                and for each $v_i\in X\setminus \{v_1\}$, $\ucn(v_i)=v$. 
                Color 4 is the free color. 
                
            \item    $|K_1|\geq 2$. 
            
            \begin{itemize}
                \item There exists a vertex $v'\in K_1\setminus \{v\}$, 
                such that $v_2\notin N(v')$.

                We assign 
                $C(v)=2$, $C(w)=3$, $C(v')=d+1$ and the vertices (if any) in $K_1\setminus\{v,v'\}$ the color 4. 
                
                We have that $\ucn(w)=v_1$, $\ucn(v_1)=w$, 
                and for each $v_i\in X\setminus \{v_1\}$, $\ucn(v_i)=v$. 
                Also $\ucn
                (v')=v$ and for each $y\in K_1\setminus \{v'\}$, $\ucn(y)=v'$.
                Color 4 is the free color. 
   
                \item All the vertices in $K_1$ are adjacent to $v_2$. 
                
                Choose a vertex $v'\in  K_1\setminus \{v\}$. Assign 
                $C(w)=4$, $C(v)=1$, $C(v')=3$ and the vertices (if any) in $K_1\setminus\{v,v'\}$ the color $d+1$.  
                
                We have that $\ucn(w)=v_1$, $\ucn(v_2)=v'$, 
                and for each $v_i\in X\setminus \{v_2\}$, $\ucn(v_i)=v$. 
                Also for each $y\in K_1$, $\ucn(y)=v_2$.
                Color 4 is the free color. 
            \end{itemize}
                
       \end{itemize} 

\noindent
 \textbf{Subcase 1.4:} There exists a clique $K_2 \neq K_1$, such that 
$|K_2|\geq 2$. 

Since $G$ is connected, there is an edge between $X$ and $K_2$. WLOG, we assume
that $v_1$ is adjacent to $w \in K_2$.
We now divide the cases based on the size of the clique  $K_1$. 

\begin{itemize}
    \item $|K_1| =1$. 
    \begin{itemize}
        \item There exists a vertex 
$w'\in K_2\setminus \{w\}$ such that 
$v_3\notin N(w')$.

We assign $C(v)=2$, $C(w)=3$, $C(w')=d+1$, 
and the vertices (if any) in $K_2\setminus \{w,w'\}$ the color 4. 
   
We have that $\ucn(w')=w$, and 
for each $y\in K_2\setminus \{w'\}$,   $\ucn(y)=w'$. 
Also $\ucn(v)=v_1$,  $\ucn(v_1)=w$, 
and for each $v_i\in X\setminus \{v_1\}$, 
$\ucn(v_i)=v$. 
Color 4 is the free color.

\item All the vertices in $K_2\setminus \{w\}$ are adjacent to  $v_3$. 
        
 We assign $C(v)=2$, $C(w)=4$, and the vertices in $K_2\setminus \{w\}$ the color $d+1$. 

We have that $\ucn(v)= \ucn(w)=v_1$ and 
for each $y\in K_2\setminus \{w\}$,   $\ucn(y)=v_3$. 
Also $\ucn(v_1)=w$, and for each $v_i\in X\setminus \{v_1\}$, 
$\ucn(v_i)=v$. 
Color 1 is the free color.
    \end{itemize}
    
    \item $|K_1|\geq 2$. 
    
    \begin{itemize}
        \item There exists a vertex $v'\in K_1\setminus \{v\}$ such that
$v_2\notin N(v')$. 

        There are two subcases based on the neighborhood of the vertices in $K_2$.

\begin{itemize}
    \item There exists a vertex 
$w'\in K_2\setminus \{w\}$ such that 
$v_3\notin N(w')$.

We assign $C(v)=2$, $C(v')=d+1$, $C(w)=3$, $C(w')=d+1$ and the vertices (if any) in $K_1\setminus \{v,v'\}$ and $K_2\setminus \{w,w'\}$ the color 4.

We have that $\ucn(v')=v$, and 
for each $y\in K_1\setminus \{v'\}$, $\ucn(y)=v'$.    
Also $\ucn(w')=w$, and 
for each $y\in K_2\setminus \{w'\}$,   $\ucn(y)=w'$. 
Moreover, $\ucn(v_1)=w$, 
and for each $v_i\in X\setminus \{v_1\}$, 
$\ucn(v_i)=v$. 
Color 4 is the free color.
            
\item All the vertices in $K_2\setminus \{w\}$ are adjacent to  $v_3$. 
            
            We assign $C(v)=2$, $C(v')=d+1$, $C(w)=4$, the vertices (if any)  
            in $K_1\setminus \{v,v'\}$ the color  1 and the vertices  
            in $K_2\setminus \{w\}$ the color $d+1$.

We have that $\ucn(v')=v$, and 
for each $y\in K_1\setminus \{v'\}$, $\ucn(y)=v'$.    
Also $\ucn(w)=v_1$, and 
for each $y\in K_2\setminus \{w\}$,   $\ucn(y)=v_3$. 
Moreover, $\ucn(v_1)=w$, 
and for each $v_i\in X\setminus \{v_1\}$, 
$\ucn(v_i)=v$. 
Color 1 is the free color.
\end{itemize}

\item All the vertices in $K_1$ are adjacent to $v_2$. 

We choose $v' \in K_1 \setminus \{v\}$ arbitrarily. 
We assign $C(v)=1$, $C(v')=3$ and 
the vertices (if any) in $K_1\setminus \{v,v'\}$ 
the color 4. 
We leave the clique $K_2$ uncolored for now.

We have that  
for each $y\in K_1$, $\ucn(y)=v_2$. 
Moreover, $\ucn(v_2)=v'$, 
and for each $v_i\in X\setminus \{v_2\}$, 
$\ucn(v_i)=v$. 
Color 4 is the free color.



    \end{itemize}

\end{itemize}

\noindent
\textbf{Case 2:}
For all  $y \in V \setminus X$, we have $\deg_X(y) < |X|$. 
    And there is a vertex
    $v\in V\setminus X$ such that $\deg_X(v)\geq  2$. 
    
    Let $v\in K$ where $K$ is a clique in $G[V \setminus X]$. 
        Recall that since $G[X]$ is 1-regular, it follows that $|X|\geq 4$. 
        We have the following cases depending on the neighborhood of $v$. 
  \begin{itemize}      
 
            \item \textbf{$N(v)\cap X$ contains a pair of adjacent vertices.} 
            
            WLOG, let $v_3,v_4\in N(v)$. Since 
            $\deg_X(v) < |X|$, 
            WLOG we assume $v_1\notin N(v)$.
            
\begin{itemize}
    \item  $|K|=1$.

                 We assign $C(v)=2$. 
                 We get 
                 that $\ucn(v)=v_3$ and 
                 $\ucn(v_3)= \ucn(v_4) = v$.
                 For each $v_i \in X \setminus \{v_3, v_4\}$, the uniquely 
                 colored neighbor  is the lone neighbor of $v_i$ in $X$.
                 The color 3 is the free color.

           \item  $|K|\geq 2$. 
           \begin{itemize}

                    \item There exists a vertex $v'\in K\setminus \{v\}$ such that 
                   $v_2\notin N(v')$.

                   We assign $C(v)=2$, $C(v')=d+1$ and the 
                   vertices (if any) in $K\setminus \{v,v'\}$ 
                   the color 4. 
                   
                   We have that $\ucn(v')=v$, and 
                   for all $y\in K\setminus \{v'\}$, $\ucn(y)=v'$. 
                   Also $\ucn(v_3) = \ucn(v_4) = v$, 
                   and for each $v_i\in X \setminus \{v_3, v_4\}$, 
                   the uniquely 
                 colored neighbor is the lone neighbor of $v_i$ in $X$.
                   We have color 4 as the free color.

                \item All vertices in $K\setminus \{v\}$ are adjacent to $v_2$. 
                
                Let $v'$ be arbitrarily chosen from $K \setminus \{v\}$. 
                We assign $C(v)=1$, $C(v')=3$ and the vertices (if any) in $K\setminus \{v,v'\}$ 
                the color $d+1$. 
                
                We have that $\ucn(v) = v_4$, and for all $y\in K \setminus \{v\}$, $\ucn(y)=v_2$. 
                Also $\ucn(v_2) = v'$, 
                  $\ucn(v_3)= \ucn(v_4)=v$, 
                   and for all $v_i\in X\setminus \{v_2, v_3, v_4\}$, 
                   the uniquely 
                 colored neighbor is the lone neighbor of $v_i$ in $X$.
                   We have  color 4 as the free color.

    \end{itemize}

            \end{itemize}
            \item  \textbf{None of the vertices in $N(v) \cap X$ are adjacent to each other.} 
            
            WLOG $v_2,v_4\in N(v)$ which implies 
            $v_1, v_3\notin N(v)$. 
             
           \begin{itemize}
               \item $|K|=1$. 
               
               We assign $C(v)=2$. 
                We get 
                 that $\ucn(v)=v_2$ and 
                 $\ucn(v_2)=\ucn(v_4) = v$.
                 For each $v_i \in X \setminus \{v_2, v_4\}$, the uniquely 
                 colored neighbor is the lone neighbor of $v_i$ in $X$.
                 The color 3 is the free color.
    
    \item $|K|\geq 2$. 
    
    \begin{itemize}
                \item There exists a $v'\in K\setminus \{v\}$ such that 
            $v_2\notin N(v')$.

            We assign $C(v)=2$, $C(v')=d+1$ and 
           the vertices (if any) in $K \setminus \{v, v'\}$ the color 3.
            
            We get that 
            $\ucn(v') = v$ and for all $y \in K \setminus \{v'\}$,
            $\ucn(y) = v'$.
            Moreover, $\ucn(v_2)=\ucn(v_4) = v$, and 
            for each $v_i \in X \setminus \{v_2, v_4\}$, the uniquely 
                 colored neighbor is the lone neighbor of $v_i$ in $X$.
            The color 3 is the free color.
           
            \item Every vertex in $K$ is adjacent to $v_2$. 
            
            We assign $C(v)=4$ and the vertices in
            $K \setminus \{v\}$ the color $d+1$. 
            
            We get that for all $y \in K$, $\ucn(y) = v_2$.
            Moreover, $\ucn(v_2)=\ucn(v_4) = v$, and 
            for each $v_i \in X \setminus \{v_2, v_4\}$, the uniquely 
                 colored neighbor is the lone neighbor of $v_i$ in $X$.
            The color 3 is the free color.
            
            \end{itemize}

           \end{itemize}

        \end{itemize}

        \noindent
        \textbf{Case 3:} Every vertex $v \in V\setminus X$ has $\deg_X(v)\leq 1$. 
        
        Since $G$ is connected, there is a $v \in V\setminus X$ such that $\deg_X(v) = 1$. Let $v \in K$ where $K$ is a clique in $G[V\setminus X]$. 
        WLOG let $N(v)\cap X=\{v_1\}$. 
        
        \begin{itemize}
            \item $|K|=1$. 
            
            We assign $C(v)=1$. We get that $\ucn(v) = v_1$.
           Also, $\ucn(v_1)=v$ and 
                 for each $v_i \in X \setminus \{v_1\}$, the uniquely 
                 colored neighbor is the lone neighbor of 
                 $v_i$ in $X$. 
                 The color 2 is  the free color.

        \item $|K|\geq 2$. 
        
       \begin{itemize}
            \item There exists a vertex $v'\in K\setminus \{v\}$ 
            such that 
            $v_1\notin N(v')$.

            We assign $C(v)=1$, $C(v')=d+1$ and the vertices (if any) in $K\setminus \{v,v'\}$ the color 2. 
 
            We get that $\ucn(v')=v$ and for all $y\in K\setminus \{v'\}$, $\ucn(y)=v'$.
                 Moreover, $\ucn(v_1)=v$, and for 
                  each $v_i \in X \setminus \{v_1 \}$, the uniquely 
                 colored neighbor is the lone neighbor of 
                 $v_i$ in $X$. 
                 The color 2 acts as the free color.

            \item All vertices in $K$ are adjacent to $v_1$.

            We assign $C(v)=3$ and 
            the vertices in $K\setminus \{v\}$ the color $d+1$. 
            
            We get that for all $y \in K$, $\ucn(y) = v_1$.
            Moreover, $\ucn(v_1)=v$, and
                 for each $v_i \in X \setminus \{v_1\}$, the uniquely 
                 colored neighbor is the lone neighbor of 
                 $v_i$ in $X$. 
                 The color 2 acts as the free color. 
 
     \end{itemize}    
        
        \end{itemize}
        
        
        In Subcase 1.1, $G$ has been assigned
        a full CFON coloring using $d+1$ colors. In all
        the other cases (and subcases therein),
         $G$ has been partially colored satisfying the rules of Lemma \ref{lem:isolated_in_X}.
        The uncolored cliques $K\in G[V\setminus X]$ can be colored with 
        the application of Lemma \ref{lem:isolated_in_X}, yielding a
        full CFON coloring of $G$ that uses $d+1$ colors.


\subsection{Proof of Lemma \ref{lem:iso_perfect_matching}}\label{sec:zerone}


Let $X = \{v_1, v_2, \dots, v_d\}$ and 
$Y = \{v_i \in X : \deg_X (v_i)\geq 1\}$.
By the conditions in the statement of the lemma, we have
$X\setminus Y \neq \emptyset$ and for each vertex $v\in Y$, $\deg_X(v)= \deg_Y(v)=1$. For each vertex $v_i\in X$, we assign the color $C(v_i)=i$.

\textbf{High Level Idea:} We have four cases depending on
how the vertices in $V\setminus X$ interact with $X\setminus Y$
and $Y$. 
In each case, we choose a vertex $v\in K$ for some clique $K \subseteq G[V\setminus X]$. 
We assign colors to the vertices in $K$ such that 
all the vertices in $K$ and $N(v)\cap X$ 
have a uniquely colored neighbor,  
while satisfying the rules of Lemma \ref{lem:isolated_in_X}. 
In particular, we identify a free color from the above partial coloring.
We use Lemma \ref{lem:isolated_in_X} to color the remaining vertices and obtain a \cf{} coloring of $G$. 

The key obstacle here is that while coloring the clique $K$, we could end up assigning the free color or the color $d+1$ to multiple vertices of $K$. 
There could exist vertices $v_i\in X\setminus Y$, such that 
$N(v_i)\subseteq K$ and all the vertices in $N(v_i)$
 are assigned the free color and the color $d+1$. 
This may leave the vertex $v_i$ without a uniquely colored neighbor. Hence, while coloring $K$, we need to handle these 
vertices separately. Let $S_K$ be the set of such vertices.

Formally, 
$S_K = \{v_i \in X \setminus Y : N(v_i) \subseteq K\} \setminus N(v)$.
The vertices in $N(v) \cap (X \setminus Y)$ rely on $v$ for their uniquely colored 
neighbor and hence does not require special attention.

Lemma \ref{lem:handlesk} shows that we can color $K$ in such a way that
all the vertices in $S_K$ have a uniquely colored neighbor,
and satisfying all the rules of Lemma \ref{lem:isolated_in_X}. Lemma 
\ref{lem:handlesk} will be proved after completing the proof of Lemma
\ref{lem:iso_perfect_matching}. For now,
we shall assume Lemma \ref{lem:handlesk} and proceed. 

\newcounter{routine}
\setcounter{routine}{\value{theorem}}
\begin{lemma}\label{lem:handlesk}
Let $G = (V,E), X, Y$ be as above. Let $v \in K$ where $K$ is a 
clique in $G[V \setminus X]$ such that $|K| \geq 2$. 
Let $C(v_i) = i$ for all $v_i \in X$ and all the vertices in $K \setminus \{v\}$ are uncolored. 
Suppose $C(v)$ is assigned and the free color $f$ is identified, in such a way that $v$ relies on a color
other than $f$ as the unique color in its neighborhood.
Then $K$ can be colored in such a way that all the vertices in $S_K$ have a 
uniquely colored neighbor, and satisfying all the rules of Lemma \ref{lem:isolated_in_X}. 
\end{lemma}
It will be convenient to denote an application of Lemma \ref{lem:handlesk}
by the 4-tuple, $(v, C(v), f, K)$. For example, we will say ``applying
Lemma \ref{lem:handlesk} to $(v, 1, 3, K)$'' to denote an application of 
Lemma \ref{lem:handlesk} where $v \in K$, $C(v) = 1$ and 3 is the free color.

We have four cases based on the neighborhoods of the vertices in $V\setminus X$.

\noindent
\textbf{Case 1: There exists a vertex $v\in V\setminus X$ such that $|N(v) \cap (X\setminus Y)| \geq 2$.}

Let $v \in K$, where $K$ is a clique in  $G[V \setminus X]$.
WLOG let $v_1, v_2\in X\setminus Y$ such that 
$v_1, v_2\in N(v)$. 

\begin{itemize}
    \item $|K|=1$. 
    
    We assign $C(v)=1$ and we get that $\ucn(v)=v_1$, 
    $\ucn(v_1)=\ucn(v_2)=v$. 
     We get the color 2 as the free color. 
     
     \item $|K|\geq 2$. 

\begin{itemize}

\item $S_K \neq \emptyset$.

We assign $C(v)=1$ and we get that $\ucn(v)=v_1$
    and for all $v_i \in N(v) \cap (X \setminus Y)$, we have $\ucn(v_i) = v$.
We now apply Lemma \ref{lem:handlesk} to  $(v,1, 2, K)$ ensuring that 
$K$ is colored, while taking the remaining vertices of $S_K$ into account.
Color 2 is the free color. 

\item $S_K=\emptyset$.

\begin{itemize}
    \item   There exists a vertex $v'\in K\setminus \{v\}$ 
    such that 
     $v_1\notin N(v')$. 

    Assign $C(v)=1$, $C(v')=d+1$ and the vertices (if any) in $K\setminus \{v,v'\}$ 
     the color $2$. 
    
    We get that $\ucn(v')=v$ and 
    for all $y\in K\setminus \{v'\}$, $\ucn(y)=v'$. 
    Also
    $\ucn(v_1)=\ucn(v_2)=v$. 
    Color 2 is the free color. 

    \item Every vertex in $K$ is adjacent to $v_1$. 
    
    Assign $C(v) = 2$ and the 
    vertices in $K \setminus \{v\}$ the color $d+1$. 
        
    We get that $\ucn(y)=v_1$ for all $y\in K$. 
    Also 
    $\ucn(v_1)=\ucn(v_2)=v$. 
    Color 1 is the free color. 
\end{itemize}

\end{itemize}
\end{itemize}
    In all the above cases, for each $v_i \in Y $, the uniquely 
    colored neighbor is the lone neighbor of 
    $v_i$ in $X$.

\noindent
\textbf{Case 2: There exists a vertex $v\in V\setminus X$ 
such that $|N(v) \cap (X \setminus Y)| = 1$ and 
$N(v) \cap Y$ contains a pair of vertices that are adjacent to each other.} 

Since Case 1 is already addressed, we assume for all $y \in V \setminus X$, 
we have $|N(y) \cap (X \setminus Y)| \leq 1$.

 Let $v \in K$, where $K$ is a clique in  $G[V \setminus X]$. 
 WLOG let $N(v)\cap (X\setminus Y) = \{v_1\}$, and let $v_2,v_3\in N(v)\cap Y$ 
         such that         $\{v_2,v_3\}\in E(G)$.

       \begin{itemize}

         \item $|K|=1$. 
         
    We assign $C(v)=1$. 
    We get that $\ucn(v)=v_1$ and 
    $\ucn(v_1)=\ucn(v_2)=\ucn(v_3)=v$. 
    We have the color 2 as the free color.

\item $|K| = 2$.

   Let $K=\{v,v'\}$. We have the following cases. 
        \begin{itemize}

                    \item $S_{K}\neq \emptyset$. 
                    
            We assign $C(v)=1$ and we get that $\ucn(v) = v_1$, $\ucn(v_1)=\ucn(v_2)=\ucn(v_3)=v$. 
            We apply Lemma \ref{lem:handlesk} 
            to $(v, 1, 2, K)$. 
            We have the color 2 as the free color. 
                
        \item $S_{K}=\emptyset$. 
                
     \textbf{Subcase 1:} $v_1 \notin N(v')$. 
            
        We assign $C(v)=1$ and $C(v')=d+1$. 
        We get that $\ucn(v)=v_1$, $\ucn(v')=v$ and
        $\ucn(v_1)=\ucn(v_2)=\ucn(v_3)=v$.  We have color 2 as the free color.

%

        \textbf{Subcase 2:} $v_1 \in N(v')$. That is, $N(v') \cap (X \setminus Y) = \{v_1\}$.
        
        We first check if there exists a clique $\widehat K \subseteq G[V \setminus X]$ such that 
        $N(v_2) \cap \widehat K \neq \emptyset$ or
        $N(v_3) \cap \widehat K \neq \emptyset$.
        
        If there is no such clique $\widehat K$, 
        we reassign $C(v_3)=2$, 
        assign $C(v)=3$ and $C(v')=d+1$. 
        We get that 
        $\ucn(v)=v_1$ and 
        $\ucn(v_1)=\ucn(v_2)=\ucn(v_3)=\ucn(v')=v$.  
        Color 2 is the free color. 
        
        Else, there exists a clique $\widehat K$ such that $N(v_2) \cap \widehat K \neq \emptyset$ or
        $N(v_3) \cap \widehat K \neq \emptyset$.
         WLOG let  $N(v_3) \cap \widehat K \neq \emptyset$.
        We assign $C(v)=2$ and $C(v')=d+1$. 
        Now the vertex $v_3$ does not have a uniquely colored neighbor. 
        Let $w\in  N(v_3) \cap \widehat K$.

      If $S_{\widehat K}\neq \emptyset$, we assign $C(w) = 3$. We have $\ucn(v)=\ucn(v')=v_1$, $\ucn(v_1)=\ucn(v_2)=v$, $\ucn(v_3)=w$ and $\ucn(w) = v_3$. 
         Due to the case definition, $|N(w) \cap (X \setminus Y)| \leq 1$. 
        For the lone vertex $v_i \in N(w) \cap (X \setminus Y)$ (if it exists), 
        we have $\ucn(v_i) = w$.
      We now apply Lemma \ref{lem:handlesk} to $(w,3, 1, \widehat K)$ to color 
      the remaining vertices $\widehat K$ taking care of the vertices in $S_{\widehat K}$. 
      We have color 1 as the free color.

      Else if $S_{\widehat K} =  \emptyset$, we do the following\footnote{One may wonder
      about the possibility of vertices $v_i \in X \setminus Y$ such that
      $N(v_i) \subseteq K \cup \widehat K$, and be concerned that these vertices $v_i$ do 
      not feature in $S_K$ or $S_{\widehat K}$. We note that there are no such vertices $v_i$. 
      This is because, we have  $N(v)\cap (X\setminus Y) = N(v')\cap (X\setminus Y) = \{v_1\}$ in order to be in Subcase 2. }. 
      \begin{itemize}         
        \item There exists a vertex $w'\in \widehat K \setminus \{w\}$, 
        such that $v_3 \notin N(w')$.  

        We assign $C(w)=3$, $C(w')=d+1$ and the vertices (if any) in $\widehat K \setminus \{w,w'\}$ the 
        color 1. 
        
        We get that $\ucn(v)=\ucn(v')=v_1$, $\ucn(v_1)=\ucn(v_2)=v$, $\ucn(v_3)=w$, $\ucn(w')=w$ and for all 
        vertices $x\in \widehat K \setminus \{w'\}$, 
        $U(x)=w'$. 
        We have the color 1 as the free color.           
            
    \item For each $x\in \widehat K$, we have $v_3 \in N(x)$.
                
                We assign $C(w)=1$ and the rest 
                of the vertices (if any) in $\widehat K\setminus \{w\}$ 
                the color $d+1$.

        We get that $\ucn(v)=\ucn(v')=v_1$,  $\ucn(v_1)=\ucn(v_2)=v$,  $\ucn(v_3)=w$, 
        and for all vertices $x\in \widehat K$, $U(x)=v_3$. 
       Color 3 is the free color. 
     \end{itemize}  
          \end{itemize}

          \item $|K|\geq 3$. 
          
          \begin{itemize}
              \item $S_{K}\neq \emptyset$. 
              
          We assign $C(v)=1$ and 
          we get that $\ucn(v) = v_1$, $\ucn(v_1)=\ucn(v_2)=\ucn(v_3)=v$. 
          We apply Lemma \ref{lem:handlesk} to $(v,1, 2, K)$.  
          We have the color 2
          as the free color.

          \item $S_{K}= \emptyset$.

          \begin{itemize}
              \item 
              There exists a vertex $v'\in K\setminus \{v\}$ such that $v_1\notin N(v')$. 
              
              We assign $C(v)=1$, $C(v')=d+1$ and 
                the vertices in 
                $K\setminus \{v,v'\}$ the color 3.

                We get that $\ucn(v')=v$ and for all $y\in K\setminus \{v'\}$, $\ucn(y)=v'$.   Also $\ucn(v_1)=\ucn(v_2)=\ucn(v_3)= v$.  
         We have the color 3 as the free color. 
              
              \item Every vertex in $K$ is adjacent to $v_1$. 
              
         Choose two vertices $v',v''\in K \setminus \{v\}$ 
            and assign $C(v) = 1$, $C(v')=2$, $C(v'')=d+1$ 
            and the vertices (if any) in $K\setminus \{v,v',v''\}$ the 
            color $3$.

            We get that $\ucn(v'')=v'$ and for all $y\in K\setminus \{v''\}$, $\ucn(y)=v''$.   
                Also $\ucn(v_1)=\ucn(v_2)=\ucn(v_3)= v$.  
         We have the color 3 as the free color. 

          \end{itemize}
          
    \end{itemize}

          \end{itemize}
         In each of the above cases, for each $v_i \in Y \setminus \{v_2, v_3\}$, the uniquely 
        colored neighbor is the lone neighbor of 
        $v_i$ in $X$.

\noindent                    
\textbf{Case 3: There exists a vertex $v\in V\setminus X$ such that 
$|N(v)\cap (X\setminus Y)| = 1$ and $|N(v)\cap Y| \geq 1$. Moreover, 
none of the vertices in  $N(v)\cap Y$ are adjacent to each other.}

Let $v\in K$ for a clique $K \subseteq G[V \setminus X]$. 
WLOG let $v_1\in N(v)\cap (X\setminus Y)$ and $v_2\in N(v)\cap Y$. 
Let $v_3$ be the lone neighbor of $v_2$ in $Y$.
It follows that $v_3\notin N(v)$.

\begin{itemize}
    \item $|K| = 1$.

        We assign $C(v)=1$ and we get that $\ucn(v)=v_1$, 
        $\ucn(v_1)=\ucn(v_2)=v$ and  $\ucn(v_3)=v_2$. 
        Color 3 is the free color.
        
    \item $|K| \geq 2$.
    

        \begin{itemize}

            \item $S_K\neq \emptyset$.

            We assign $C(v)=1$  and we get that $\ucn(v)=v_1$, 
        $\ucn(v_1)=\ucn(v_2)=v$ and  $\ucn(v_3)=v_2$.
            We apply Lemma \ref{lem:handlesk} to
            $(v,1, 3, K)$. Color 3 is the free color.

            \item $S_K = \emptyset$.
            
            \begin{itemize}

    \item There exists a vertex $v'\in K\setminus \{v\}$ such that $v_2\in N(v')$.

    We assign 
    $C(v)=1$, $C(v')=d+1$ and the vertices (if any) in $K\setminus \{v,v'\}$ the color 
    3. 
    
    We get that $\ucn(v')=v_2$ and 
    for all $y\in K\setminus \{v'\}$, $\ucn(y)=v'$. 
    Also
    $\ucn(v_1)=\ucn(v_2)=v$ and  $\ucn(v_3)=v_2$. 
    We have  color 3 as the free color. 
    
    \item  None of the vertices in $K\setminus \{v\}$ are adjacent to $v_2$.
    
    We assign 
    $C(v)=2$ and the vertices in $K\setminus \{v\}$ the color $d+1$. 
    We get that $\ucn(v)=v_1$ and 
    for all $y\in K\setminus \{v\}$, $\ucn(y)=v$. 
    Also
    $\ucn(v_1)=\ucn(v_2)=v$ and  $\ucn(v_3)=v_2$. 
    We have color 3 as the free color. 
    \end{itemize}
\end{itemize}    
 
    \end{itemize}
In each of the above cases, for each $v_i \in Y \setminus \{v_2, v_3\}$, the uniquely 
    colored neighbor is the lone neighbor of 
    $v_i$ in $X$. 

\noindent
\textbf{Case 4: For each  $y \in V \setminus X$ such that 
$|N(y) \cap (X\setminus Y)| = 1$, we have  
$|N(y) \cap Y| = 0$.} 

Since Case 1 is addressed, we assume that each vertex in $V\setminus X$ has at most 1 
neighbor in $X\setminus Y$.

Since $G$ is connected and since\footnote{This is where we make use of the assumption that $X \setminus Y$ is nonempty.} 
$|X \setminus Y| \geq 1$, we can choose a clique $K\subseteq G[V\setminus X]$ 
with distinct vertices\footnote{Because of the definition of Case 4, it follows that $v \neq v'$.} 
$v, v' \in K$ such that
$N(v) \cap (X\setminus Y) \neq \emptyset$ and $N(v') \cap Y \neq \emptyset$.  
WLOG let $v_1\in N(v) \cap (X\setminus Y)$ 
 and $v_2\in N(v')\cap Y$. Let $v_3$ be the lone neighbor of
 $v_2$ in $Y$. It follows that $v_2, v_3 \notin N(v)$.

\begin{itemize}

        \item $S_K =\emptyset$. 

We assign $C(v)=3$, $C(v')=d+1$ and 
the vertices (if any) in $K\setminus \{v,v'\}$ the color 1. 
We get that $\ucn(v')=v_2$ and 
for all $y\in K\setminus \{v'\}$, $\ucn(y)=v'$. 
Also $\ucn(v_1)=v$, $\ucn(v_2)=v_3$ and $\ucn(v_3)=v_2$. 
Color 1 is the free color. 

    \item $S_K \neq \emptyset$. 
    
    We 
assign $C(v)=3$ and we have 
$\ucn(v_1) = v$, $\ucn(v_2) = v_3$ and $\ucn(v_3) = v_2$.

Recall that $S_K = \{v_i \in X \setminus Y : N(v_i) \subseteq K\} \setminus N(v)$.
    For each $v_i\in S_K$, choose a vertex 
    $w_i\in N(v_i)$, assign $C(w_i)=i$ and let $\ucn(v_i)=w_i$. 
    Since each vertex in $V\setminus X$ has at most 1 
neighbor in $X\setminus Y$, it also follows that $N(v_i) \cap N(v_{i'}) = \emptyset$
    for any two vertices $v_i, v_{i'} \in S_K$.
    
    Because of the condition of Case 4, $N(v') \cap (X\setminus Y) = \emptyset$. Assign $C(v') = d+1$,
    and assign the color 1 to all the remaining uncolored vertices (if any) in $K$.
    We have $\ucn(v') = v_2$
    and for all $y\in K\setminus \{v'\}$, $\ucn(y)=v'$.
    Color 1 is the free color\footnote{One may wonder why we did not apply
    Lemma \ref{lem:handlesk} to $(v, 3, 1, K)$ in this situation.
    This is because Lemma \ref{lem:handlesk} requires $v$ to rely on 
    a color other than the free color as the unique color in its neighborhood. There is no assignment that meets this requirement.}.

\end{itemize}
In each of the above cases, for each $v_i \in Y \setminus \{v_2, v_3\}$, the uniquely 
    colored neighbor is the lone neighbor of 
    $v_i$ in $X$.

 We have concluded the four cases. In each of the cases we have a free color $f$. We use Lemma \ref{lem:isolated_in_X} to get a uniquely colored neighbor for 
 remaining vertices in $X\setminus Y$, the cliques in $G[V \setminus X]$ 
 and thereby obtain a CFON coloring.


\newcounter{thmsave}
\setcounter{thmsave}{\value{theorem}}
\setcounter{theorem}{\value{routine}}
\begin{lemma}[Restated]
Let $G = (V,E)$ be a graph and $X =\{v_1, v_2, \ldots, v_d\} \subseteq V$ be a set of vertices such that $G[V\setminus X]$ is a disjoint union of cliques. 
Let $Y = \{v_i \in X : \deg_X (v_i)\geq 1\}$. 
Let $v \in K$ where $K$ is a 
clique in $G[V \setminus X]$ such that $|K| \geq 2$. 
Let $C(v_i) = i$ for all $v_i \in X$ and all the vertices in $K \setminus \{v\}$ are uncolored. Suppose $C(v)$ is assigned and the free color $f$ is identified, in such a way that $v$ relies on a color
other than $f$ as the unique color in its neighborhood.
Then $K$ can be colored  
 in such a way that all the vertices in $S_K$ have a 
uniquely colored neighbor, and satisfying all the rules of Lemma \ref{lem:isolated_in_X}. 
\end{lemma}
\setcounter{theorem}{\value{thmsave}}

\begin{proof}
Recall that $S_K = \{v_i \in X \setminus Y : N(v_i) \subseteq K\}  \setminus N(v)$.
Let $S_K=\{v_{j_1}, v_{j_2}, \cdots, v_{j_m}\}$, for some $m \geq 1$.
    For each $v_i\in S_K$, choose an uncolored vertex 
    $w_i\in N(v_i)$, assign $C(w_i)=i$ and let $\ucn(v_i)=w_i$. 
    If all the vertices in $N(v_i)$ 
    are colored, 
    we arbitrarily choose a vertex in $N(v_i)$
    as $\ucn(v_i)$.
    WLOG let the colors used 
    in $K$ because of the above process
    be $\{j_1,j_2,\cdots, j_{m'}\}$ where $m'\leq m$. 
    Note that the vertex $v$ is colored prior to the application of this lemma,
    and has a uniquely colored neighbor as well. Hence we do not talk about
    $C(v)$ and $\ucn(v)$ in this proof.
    
    We have the following cases based on the number of uncolored vertices in $K$. 
        
\begin{itemize}
    \item All the vertices in $K$ are colored. 
    
    Each vertex in $K \setminus \{v\}$ was colored because it was chosen as $w_i$ by some $v_i \in S_K$. Hence $\ucn(w_i) = v_i$.

    \item $K$ contains exactly one   uncolored vertex.

    Let the uncolored vertex in $K$ be $v'$. 
    If $v'$ has a uniquely colored neighbor, 
    we assign $C(v')=d+1$. 
    Now, every vertex in $K$ has a uniquely colored neighbor. 
    
    If $v'$ does not have a 
    uniquely colored neighbor, 
    we have two cases depending on the number of colors used in $K$. 

    \begin{itemize}
        \item  $m=m'$. 
        
        This means that each vertex $v_i\in S_K$ chose a neighbor $w_i\in K$ and assigned the color $i$ to it. 
        So 
        $v'$ sees each of the 
        the colors $j_1, j_2, \cdots, j_{m}$ twice in its neighborhood.  
        This means that $v_{j_1}, v_{j_2}, \cdots, v_{j_{m}}\in N(v')$. 
        Recall that $f$ is the free color, where $1 \leq f \leq d$,
        and hence $v_f \in X$ sees a color other than $f$ as the unique color in its 
        neighborhood.
        We do the following to obtain a uniquely colored neighbor for $v'$: 
        
        \begin{itemize}
            \item There exists a vertex $v''\in K\setminus \{v, v'\}$ such that 
            $v_f\in N(v'')\cap X$. 
            
            Let $C(v'')=k$ due to a vertex 
            $v_k\in N(v'')\cap S_K$. 
            Assign $C(v')=k$ and 
            reassign $C(v'')=d+1$. 
            We have that $\ucn(v')=v_k$, $\ucn(v'')=v_f$ 
           and we reassign $\ucn(v_k) = v'$.
            

            \item None of the vertices in $K\setminus \{v, v'\}$ is adjacent to $v_f$. 

            Note that $v_f\notin N(v')$, else $v_f$ would have served as a uniquely colored neighbor for $v'$.

            Choose a vertex $v''\in K\setminus \{v, v'\}$. 
            Suppose $C(v'')=k$ and this implies that 
            $v_k\in N(v'')\cap S_K$. 
            We reassign $C(v'')=f$ and assign $C(v')=d+1$. 
            We have $\ucn(v') = v_k$ and we reassign $\ucn(v_k)=v'$. 
            The assignment of $d+1$ as the unique color in the neighborhood of
            $v_k$ is an 
            exception. However, this is fine as $N(v_k)$ is contained in $K$, 
            and does not interact with any other cliques in $G[V \setminus X]$.
        \end{itemize}
        \item $m>m'$. 
        
        This implies that 
        there exists a vertex $v_j\in S_K$ such that 
        the color $j$ is not given to any vertex in $K$. 
        So $v_j$ must be seeing a vertex $v''\in K \setminus \{v'\}$ 
        as its uniquely colored neighbor. 
        
        We claim that $v_j\notin N(v')$. 
        If $v_j\in N(v')$, then $v_j$ is the lone vertex in $N(v')$
        that is colored $j$, and hence is a uniquely colored neighbor
        for $v'$. As per the scope of this case, $v'$ does not have a
        uniquely colored neighbor. This is a contradiction.
        
        We reassign $C(v'')=j$ and assign $C(v')=d+1$. 
        We get that $U(v')=v''$.  
            
    \end{itemize}
    
    

    \item $K$ contains at least two uncolored vertices. 
    
    We first check if there exists an uncolored vertex $v'$ in $K$ such that $v'$ has a uniquely colored neighbor other than $v_f$. 
    If such a $v'$ exists, then we assign $C(v')=d+1$ and the remaining uncolored vertices in $K$ the free color $f$. 
    For all $w \in K$ such that $C(w) = f$, we have $\ucn(w) = v'$.
    

        If such a vertex $v'$ does not exist, we have the following 
        cases based on the relation between $m$ and $m'$. 
        \begin{itemize}
            \item $m=m'$.
            
        Choose a colored vertex $w\in K\setminus \{v\}$
        and  an uncolored vertex $v' \in K$.
        Suppose $C(w)=j$, which means that 
        $v_j\in N(w)$. Since $v'$ does not see a uniquely colored neighbor other than $v_f$, it is the case that 
        $v_j\in N(v')$. 
        
        We reassign $C(w)=f$, assign $C(v')=d+1$ and 
        the remaining uncolored vertices (if any) in $K$ 
        the free color $f$. 
        We get that $\ucn(v')=v_j$ and 
        $\ucn(v_j)=v'$. All the vertices in $K\setminus \{v'\}$ will have $v'$ as their uniquely colored neighbor. 
        
        The assignment of $d+1$ as the unique color in the neighborhood of
        $v_j$ is an 
        exception. However, this is fine as $N(v_j)$ is contained in $K$, 
        and does not interact with any other cliques in $G[V \setminus X]$.
        
        \item $m>m'$. 

This implies that there exists a vertex 
$v_j\in S_K$ such that the color $j$ was not used in $K$. This also 
implies that none of the uncolored vertices in $K$ have $v_j$ in their
neighborhood. This is because if $v_j$ had an uncolored neighbor in $K$,
then that neighbor would have been colored $j$ in the coloring process
performed at the beginning of this proof. 

We choose two uncolored vertices $v',v''\in K$ and assign $C(v')=d+1$, $C(v'')=j$ and the remaining uncolored vertices (if any) the color $f$. 

We get that $\ucn(v')=v''$ and for all other vertices $w\in K\setminus \{v'\}$ will have $U(w)=v'$. 
    \end{itemize} 
\end{itemize}
\end{proof}

\bibliographystyle{plain}

\bibliography{BibFile}
\appendix

\end{document}